\documentclass[11pt,reqno]{amsart}

\usepackage{amsfonts,amssymb,amsmath,mathrsfs,graphicx}
\usepackage{float}
\usepackage{epsfig}
\usepackage{subcaption}
\usepackage{graphicx,colortbl,here}
\usepackage{soul}
\usepackage{xcolor}
\usepackage{hyperref}
\usepackage{etoolbox}
\usepackage{tikz}
\usetikzlibrary{calc}
\usetikzlibrary{decorations.pathmorphing}
\allowdisplaybreaks
\makeatletter

\usepackage[export]{adjustbox}

\numberwithin{equation}{section}

\newcommand{\defhighlighter}[3][]{%
\tikzset{every highlighter/.style={color=#2, fill opacity=#3, #1}}%
}

\defhighlighter{yellow}{.5}

\newcommand{\highlight@DoHighlight}{
\fill [ decoration = {random steps, amplitude=1pt, segment length=15pt}
, outer sep = -15pt, inner sep = 0pt, decorate
, every highlighter, this highlighter ]
($(begin highlight)+(0,8pt)$) rectangle ($(end highlight)+(0,-3pt)$) ;
}

\newcommand{\highlight@BeginHighlight}{
\coordinate (begin highlight) at (0,0) ;
}

\newcommand{\highlight@EndHighlight}{
\coordinate (end highlight) at (0,0) ;
}

\newdimen\highlight@previous
\newdimen\highlight@current

\DeclareRobustCommand*\highlight[1][]{%
\tikzset{this highlighter/.style={#1}}%
\SOUL@setup
\def\SOUL@preamble{%
\highlight@BeginHighlight
\highlight@EndHighlight
\end{tikzpicture}%
}%
\def\SOUL@postamble{%
\begin{tikzpicture}[overlay, remember picture]
\highlight@EndHighlight
\highlight@DoHighlight
\end{tikzpicture}%
}%
\def\SOUL@everyhyphen{%
\discretionary{%
\SOUL@setkern\SOUL@hyphkern
\SOUL@sethyphenchar
\tikz[overlay, remember picture] \highlight@EndHighlight ;%
}{%
}{%
\SOUL@setkern\SOUL@charkern
}%
}%
\def\SOUL@everyexhyphen##1{%
\SOUL@setkern\SOUL@hyphkern
\hbox{##1}%
\discretionary{%
\tikz[overlay, remember picture] \highlight@EndHighlight ;%
}{%
}{%
\SOUL@setkern\SOUL@charkern
}%
}%
\def\SOUL@everysyllable{%
\begin{tikzpicture}[overlay, remember picture]
\path let \p0 = (begin highlight), \p1 = (0,0) in \pgfextra
\global\highlight@previous=\y0
\global\highlight@current =\y1
\endpgfextra (0,0) ;
\ifdim\highlight@current < \highlight@previous
\highlight@DoHighlight
\highlight@BeginHighlight
\fi
\end{tikzpicture}%
\the\SOUL@syllable
\tikz[overlay, remember picture] \highlight@EndHighlight ;%
}%
\SOUL@
}
\makeatother

\title[Finite-time flocking of the infinite CS model]{Finite-time Flocking of an Infinite Cucker-Smale Type Particle Model With Nonlinear Velocity Couplings}

\thanks{$^{*}$ Corresponding author.
	E-mail addresses:
	syha@snu.ac.kr (S.-Y. Ha),
	wangxinyumath@hit.edu.cn (X. Wang),
	fqzeng@hrbeu.edu.cn (F. Zeng).}

\author[Ha]{Seung-Yeal Ha\textsuperscript{1}}
\author[Wang]{Xinyu Wang\textsuperscript{2}}
\author[Zeng]{Fanqin Zeng\textsuperscript{3,*}}

\makeatletter
\let\original@setauthors\@setauthors
\renewcommand{\@setauthors}{%
  \original@setauthors
  \begingroup
  \centering\footnotesize\normalfont
  \textsuperscript{1}Department of Mathematical Sciences and Research Institute of Mathematics,\par
  Seoul National University, Seoul, 08826, Korea (Republic of)\par
\smallskip
  \textsuperscript{2}School of Mathematics,\par
  Harbin Institute of Technology, Harbin 150001, P. R. China\par
 \smallskip
  \textsuperscript{3}School of Mathematical Sciences,\par
  Harbin Engineering University, Harbin 150001, Heilongjiang, P.R. China\par
 
  \endgroup
}
\makeatother

\begin{document}

\newtheorem{theorem}{Theorem}[section]
\newtheorem{lemma}{Lemma}[section]
\newtheorem{corollary}{Corollary}[section]
\newtheorem{proposition}{Proposition}[section]
\newtheorem{remark}{Remark}[section]
\newtheorem{definition}{Definition}[section]
\newtheorem{example}{Example}[section]
\renewcommand{\theequation}{\thesection.\arabic{equation}}
\renewcommand{\thetheorem}{\thesection.\arabic{theorem}}
\renewcommand{\thelemma}{\thesection.\arabic{lemma}}
\newcommand{\bbr}{\mathbb R}
\newcommand{\bbz}{\mathbb Z}
\newcommand{\bbn}{\mathbb N}
\newcommand{\bbs}{\mathbb S}
\newcommand{\bbp}{\mathbb P}
\newcommand{\kp}{\kappa}
\newcommand{\mi}{\mathrm i}
\newcommand{\ddiv}{\textrm{div}}
\newcommand{\bn}{\bf n}
\newcommand{\rr}[1]{\rho_{{#1}}}
\newcommand{\thh}{\theta}

\newcommand{\di}{\mathrm d}

\def\charf {\mbox{{\text 1}\kern-.24em {\text l}}}
\renewcommand{\arraystretch}{1.5}

% Empty acknowledgments footnote omitted.

\begin{abstract}
We investigate finite-time flocking for an infinite-particle Cucker-Smale type system with nonlinear coupling under sender networks.
We first establish global existence of classical solutions for fixed sender networks, allowing for zero sender weights.
By introducing a componentwise diameter framework and exploiting sublinear dissipation
mechanisms, we derive sufficient conditions for finite-time flocking together with explicit alignment-time criteria. For fixed sender networks, we establish conditional finite-time flocking for integrable communication weights and unconditional finite-time flocking for non-integrable communication weights.
The analysis is also extended to switching sender networks. The resulting alignment-time estimates contain no explicit dependence on the number of particles when the total mass and initial diameter bounds are controlled.
\end{abstract}

\maketitle


%\tableofcontents

%%%%%%%%%%%%%%%%%%%%%%%%%%%%%%%%%%%%%%%%%%%%%%%%%%%%%%%%%%%%%%%%%%%%%%%%%%%%%%%%%%%%%%%%%%%%%%%%%%%%%%%
%
%  Section
%
%%%%%%%%%%%%%%%%%%%%%%%%%%%%%%%%%%%%%%%%%%%%%%%%%%%%%%%%%%%%%%%%%%%%%%%%%%%%%%%%%%%%%%%%%%%%%%%%%%%%%%

\section{Introduction} \label{sec:1}
\setcounter{equation}{0} 
Collective dynamics in multi-agent systems have attracted sustained attention due
to broad applications in robot coordination, formation control, biological swarms,
and social dynamics; see, e.g., \cite{B-H,C-S,D-Q,J-L-M,P-E-G,R,T-T,V-Z,Y-X-L-Z,Z-L-Z} and references therein.
Among various mathematical frameworks, the Cucker-Smale (CS) model is a canonical
paradigm for flocking and alignment, demonstrating how simple interaction rules can
generate coherent collective motion \cite{C-S,C-D}.

Most classical studies on CS-type systems focus on \emph{asymptotic} velocity
alignment, where agents converge to a common velocity as time tends to infinity
under suitable conditions on the communication weight and interaction topology
\cite{C-S,C-D,L-X,S}.
In many control-oriented and safety-critical applications, however, it is desirable
to achieve consensus or flocking within a \emph{finite} time horizon, ideally with
explicit convergence-time guarantees.
This has stimulated extensive research on finite-time and fixed-time
consensus/flocking for finite-dimensional multi-agent systems, including scenarios
with switching graphs and robustness requirements \cite{B-B,C-R,G-V-C-S-R,H-W-W-S,J-W,L-W-L-L,W-L-Y-C-S,R-L-L-W,W-Y-L,Y-W,Z-N-S-S}.

Finite-time convergence in multi-agent systems can be induced by various mechanisms, including nonlinear dissipation through non-Lipschitz or discontinuous couplings and suitable intrinsic dynamics; see, e.g., \cite{B-B,C-R,G-V-C-S-R,H-W-W-S,J-W,L-W-L-L,R-L-L-W,W-Y-L,W-L-Y-C-S}. In particular, sublinear velocity couplings provide a natural mechanism for finite-time alignment, as the relative strength of dissipation increases as velocity differences approach zero. This is in contrast to the classical linear alignment mechanism, which typically leads to asymptotic velocity alignment. For the CS model with a finite number of particles, finite-time flocking has been studied by incorporating nonlinear velocity couplings and employing energy-based methods \cite{H-H-K,S-L-L-S}. For finite-time flocking of a continuum system consisting of an uncountable number of agents, we refer to \cite{A}.

In contrast to finite-particle systems, the study of infinite-particle systems involves an infinite-dimensional phase space, which gives rise to additional difficulties in both existence theory and flocking analysis. For the infinite-particle CS model, Wang and Xue \cite{W-X} established global well-posedness and flocking estimates. Subsequently, they employed approximation arguments, weighted compactness, and energy estimates to study CS dynamics on infinite graphs \cite{W-X2}, and investigated the relation between flocking and the spectral properties of the interaction operator for infinite graphs with a central vertex group \cite{W-X3}. More recently, Zhang \cite{Z-K} established velocity convergence of individual particles under a short-range communication weight via energy estimates.

Another difficulty in the flocking analysis of infinite-particle systems is that the supremum defining the velocity diameter need not be attained. To overcome this difficulty, Ha, Wang, and Xue \cite{H-W-X} employed diameter estimates over finite index sets together with tail estimates in their study of the infinite Motsch--Tadmor model under sender networks. Finite truncation arguments were also employed in the study of the infinite relativistic CS model \cite{H-L-R-W-X}. Related analytical issues arise in the study of infinite oscillator systems. We refer to Bramburger \cite{B} for the stability of phase-locked solutions via random walks on weighted graphs and semigroup estimates, Ha, Lee, and Shim \cite{H-L-S} for synchronization and clustering of the infinite Kuramoto model via order parameters, conservation properties, and estimates near extremal phases or frequencies, and Ha, Lee, and Yoon \cite{H-L-Y} for uniform-in-time stability and the continuous-time limit of the time-discrete infinite Kuramoto model.

In this paper, we study the finite-time flocking behavior of an infinite-particle CS-type system with nonlinear velocity coupling under sender networks. While non-Lipschitz couplings can force velocity differences to vanish in finite time, they also create major analytical obstacles. The non-Lipschitz behavior near the origin prevents a direct application of the standard Picard-Lindel\"of theorem. In addition to the nonlinearity, two structural features further complicate the analysis.

First, we consider a countably infinite CS-type particle system equipped with a mass profile $(m_i)_{i\in\mathbb N}$. Let $\ell_\infty(\mathbb R^d)$ be the Banach space of bounded sequences in $\mathbb R^d$. In this setting, the natural phase space $\ell_\infty(\mathbb R^d)\times\ell_\infty(\mathbb R^d)$ is infinite-dimensional, and bounded sets are not relatively compact in the norm topology. This makes the global existence of classical solutions challenging. Moreover, the suprema and infima defining the componentwise velocity diameters need not be attained, so the usual extremal-particle arguments for finite systems cannot be applied directly.

A second difficulty arises from the structure of the interaction network. Many realistic systems exhibit directed information flow. In this work we focus on \emph{sender networks}, in which the sender weight depends only on the sending particle. This setting models broadcast-like communication and allows particles to have different levels of influence, including zero sender weights. Although the interaction coefficients are generally asymmetric, the fixed sender structure admits a weighted symmetry. More precisely, let $x_i\in\mathbb R^d$ be the position of the $i$-th particle, let $\|\cdot\|$ denote the Euclidean norm, and let $H$ be a distance-dependent communication weight. Then we have
\[
m_i m_j H(\|x_j-x_i\|)=m_j m_i H(\|x_i-x_j\|).
\] However, the resulting weighted energy estimates do not by themselves control the velocity diameter, as the sender weights have no positive lower bound. In particular, particles with zero sender weight do not contribute to the weighted energy but must still be included in the flocking analysis. We also consider switching sender networks as an extension of the fixed-network setting; see \cite{C-D,G-V-C-S-R,J-W,L-X} for flocking and consensus under switching interactions.

\smallskip
To set the stage, we now introduce the model and the notion of finite-time flocking. For each $i\in\mathbb N$, let $x_i,v_i\in\mathbb R^d$ denote the
position and velocity of the $i$-th particle. We consider the
following Cauchy problem for the infinite-particle Cucker-Smale (IPCS) model with nonlinear velocity coupling:
\begin{equation}\label{A-1}
	\begin{cases}
	\displaystyle\frac{\di x_i}{\di t}=v_i,
		\quad t>0,\quad i\in\mathbb N,\\[4pt]
		\displaystyle\frac{\di v_i}{\di t}
		=\sum_{j=1}^{\infty}m_j
		H\bigl(\|x_j-x_i\|\bigr)\Gamma(v_j-v_i),\\[4pt]
			\displaystyle x_i(0)=x_i^{\rm in},
		\quad v_i(0)=v_i^{\rm in}.
	\end{cases}
\end{equation}
The sender weights are nonnegative and have finite positive total mass:
\[
m_j\ge0,\quad
\sum_{j=1}^{\infty}m_j=M\in(0,\infty).
\]
Here and throughout the paper, $\|\cdot\|$ denotes the Euclidean
norm in $\mathbb R^d$. We impose the following assumptions on
the communication weight and velocity coupling.

\begin{itemize}
	\item[$(\mathcal A_1)$]
The communication weight $H:[0,\infty)\to[0,\infty)$ is bounded, Lipschitz continuous, and non-increasing. More precisely, with $\|H\|_\infty:=\sup_{r\ge0}H(r)$, we have
	\[
	0\le H(r)\le\|H\|_\infty<\infty,\quad
	H(r_1)\ge H(r_2)
	\quad\text{whenever }0\le r_1\le r_2.
	\]
	
	\item[$(\mathcal A_2)$]
	For some $\alpha\in(0,1)$, the nonlinear velocity coupling
	$\Gamma:\mathbb R^d\to\mathbb R^d$ is defined componentwise by
	\begin{equation}\label{A-2}
		\Gamma(s)
		=
		\left(
		\operatorname{sign}(s^{(1)})|s^{(1)}|^\alpha,
		\ldots,
		\operatorname{sign}(s^{(d)})|s^{(d)}|^\alpha
		\right)^\top,
		\quad s\in\mathbb R^d.
	\end{equation}
\end{itemize}

We use the following notion of finite-time flocking.

\begin{definition}[Finite-time flocking]
	System \eqref{A-1} is said to exhibit finite-time flocking if there exists
	$T_*\ge0$ such that:
	\begin{itemize}
		\item[(i)] Finite-time velocity alignment:
		\[
		\sup_{i,j\in\mathbb N}\|v_i(t)-v_j(t)\|=0,\quad \forall~ t\ge T_*.
		\]
		Any such time $T_*$ is referred to as a flocking time.
		\item[(ii)] Uniform boundedness of relative positions:
		\[
		\sup_{t\ge0}\sup_{i,j\in\mathbb N}\|x_i(t)-x_j(t)\|<\infty.
		\]
	\end{itemize}
\end{definition}

Once the velocities coincide, relative positions remain constant, while the ensemble may continue to translate with a common velocity. The definition does not prescribe a formation geometry or guarantee collision avoidance.

To address the above difficulties, we first establish global existence of classical solutions for initial data in $\mathcal{B}=\ell_\infty(\mathbb{R}^d)\times\ell_\infty(\mathbb{R}^d)$ by applying the Schauder--Tychonoff fixed-point theorem in a suitable Hilbert space $\tilde{\mathcal H}=\ell_{\tilde m}^2(\mathbb{R}^d)\times\ell_{\tilde m}^2(\mathbb{R}^d)$. Here, $\ell_{\tilde m}^2(\mathbb{R}^d)$ denotes the space of square-summable sequences with respect to the auxiliary weights $\tilde m_i:=m_i+2^{-i}>0$. The summability of these auxiliary weights ensures that bounded sets in $\ell_\infty(\mathbb{R}^d)$ are relatively compact in $\ell_{\tilde m}^2(\mathbb{R}^d)$, yielding a compact embedding $\mathcal{B}\hookrightarrow\tilde{\mathcal H}$. Based on this compactness and the continuity of the vector field, we construct a suitable fixed-point map to establish local existence. Suitable a priori bounds then allow us to extend the solutions globally in time; see the appendices.
To study finite-time flocking, we adopt a \emph{componentwise diameter} framework. Instead of relying on weighted energy estimates, we track the evolution of the velocity diameter in each component and derive differential inequalities in the sense of Dini derivatives. By using approximate extremizers to handle suprema and infima that need not be attained, we obtain for fixed sender weights and almost every $t>0$
\[
\mathcal D^+D^{(k)}(t)\le -MH(D_x(t))\bigl(D^{(k)}(t)\bigr)^\alpha,
\]
where $$D^{(k)}(t):=\sup_{i,j\in\mathbb N}|v_i^{(k)}(t)-v_j^{(k)}(t)|,\quad D_x(t):=\sup_{i,j\in\mathbb N}\|x_i(t)-x_j(t)\|,$$ and $\mathcal D^+$ denotes the upper right Dini derivative. This approach directly captures the finite-time decay induced by the sublinear dissipation and yields explicit upper bounds for the flocking time. In particular, we obtain unconditional finite-time flocking for non-integrable communication weights and sufficient conditions on the initial configuration and total mass for integrable communication weights. The estimates apply to all particles, including those with zero sender weight. The same diameter argument applies to finite and infinite systems, and the alignment-time bounds have no explicit dependence on the number of particles when the total mass and initial diameters are controlled.

In summary, we establish global existence of classical solutions and quantitative finite-time flocking for an infinite-particle CS-type system with non-Lipschitz velocity coupling, allowing for zero sender weights. Convergence of finite-particle approximations to a specified infinite-particle solution is not addressed. As an extension, we also derive a finite-time flocking criterion for switching sender networks in terms of cumulative effective communication.

\vspace{0.2cm}

The rest of the paper is organized as follows. In Section~\ref{sec:2}, we review several basic materials to be used in the subsequent sections. In Section~\ref{sec:3}, we derive componentwise diameter estimates and present sufficient conditions for finite-time flocking under fixed sender networks. We then extend our analysis to switching sender networks. Finally, Section~\ref{sec:6} is devoted to a brief summary of our main results. In Appendices A-C, we provide the proofs of the global existence of classical solutions and the compact embedding result.

\vspace{0.2cm}

\noindent \textbf{Notations.}~Throughout the paper, we use the following notation. For the position and velocity configurations and their initial data, we write
\[
x:=(x_i)_{i\ge1},\quad v:=(v_i)_{i\ge1},\quad x^{\rm in}:=(x_i^{\rm in})_{i\ge1},\quad v^{\rm in}:=(v_i^{\rm in})_{i\ge1}.
\] For a sequence $v=(v_i)_{i\ge1}$ with $v_i\in\mathbb R^d$, we denote
\[
\|v\|_\infty:=\sup_{i\ge1}\|v_i\|,\quad
\ell_\infty(\mathbb R^d):=\{v=(v_i)_{i\ge1}:v_i\in\mathbb R^d,\ \|v\|_\infty<\infty\}.
\]
For each $k\in[d]:=\{1,2,\dots,d\}$, we denote by $v_i^{(k)}$ the $k$-th component of $v_i$ and set $v^{(k)}:=(v_i^{(k)})_{i\ge1}$. Let $m=(m_i)_{i\in\mathbb N}$ be a sequence of nonnegative weights satisfying
\[
\sum_{i=1}^\infty m_i < \infty,\quad m_i\ge0, \quad\forall ~i\in \mathbb{N}.
\]
For a sequence $x=(x_i)_{i\in\mathbb N}$ with $x_i\in\mathbb R^d$, we define the weighted seminorm by
\[
\|x\|_{\ell_m^2}
:= \left( \sum_{i=1}^\infty m_i \|x_i\|^2 \right)^{1/2}.
\]
We then define the weighted space
\[
\ell_m^2(\mathbb R^d)
:= \bigl\{ x=(x_i)_{i\in\mathbb N}: x_i\in \mathbb R^d ;
\|x\|_{\ell_m^2} < \infty \bigr\}.
\]
Denote \[
\mathcal{B}:= \ell_\infty(\mathbb{R}^d) \times \ell_\infty(\mathbb{R}^d),
\quad \text{with}\
\|(x,v)\|_\mathcal{B}:= \|x\|_\infty + \|v\|_\infty,
\]
\[
\mathcal{H}:= \ell_m^2(\mathbb{R}^d) \times \ell_m^2(\mathbb{R}^d), 
\quad \text{with}\
\|(x,v)\|_{\mathcal{H}}:= \|x\|_{\ell_m^2}+ \|v\|_{\ell_m^2}.
\]
If some $m_i=0$, the displayed weighted quantity is only a seminorm. Set $\tilde m_i:=m_i+2^{-i}>0$ and
\[
\tilde{\mathcal H}:=\ell_{\tilde m}^2(\mathbb R^d)\times\ell_{\tilde m}^2(\mathbb R^d),\quad
\|(x,v)\|_{\tilde{\mathcal H}}:=\big(\|x\|_{\ell_{\tilde m}^2}^2+\|v\|_{\ell_{\tilde m}^2}^2\big)^{1/2}.
\]
All Hilbert-space arguments and differentiability assertions below use this strictly positive auxiliary weight. For simplicity, we write $\Gamma^{(k)}(s):=\operatorname{sign}(s)|s|^\alpha$ for $s\in\mathbb R$ and $k\in[d]$.
We denote \[
H_{ij}(t)=H(\|x_j(t)-x_{i}(t)\|),
\]
\[
M^{(k)}:=\sup_{i\ge1} v_i^{(k)},\quad m^{(k)}:=\inf_{i\ge1} v_i^{(k)},\]
\[D^{(k)}:=M^{(k)}-m^{(k)}=\sup_{i,j\ge1}|v_i^{(k)}-v_j^{(k)}|,\quad \forall \ k\in[d].
\]
We write $D_x(t):=\sup_{i,j\ge1}\|x_i(t)-x_j(t)\|$. 

\vspace{0.5cm}

\section{Preliminaries}\label{sec:2}
\setcounter{equation}{0}
In this section, we collect several preliminary properties of the IPCS dynamics that will be used throughout the paper. We first derive basic regularity and monotonicity properties of the velocity components,
which will be used in the componentwise finite-time flocking analysis.

\begin{lemma}[Non-increasing velocity diameter]\label{L2.1}
	Let $(x,v)\in C([0,T];\mathcal B)\cap C^1((0,T];\tilde{\mathcal H})$
	be a solution to system \eqref{A-1} with initial data
	$(x^{\rm in},v^{\rm in})\in \mathcal{B}$.
	Then, the following assertions hold.
	\begin{itemize}
		\item[(i)] For each $k\in[d]$, $\sup_i|\dot v^{(k)}_i(t)|$ is uniformly bounded on $t\in(0,T)$, i.e.,
		there exists a constant $C_T>0$ such that
		\begin{equation*}
	\sup_{0<t<T}\sup_i |\dot v_i^{(k)}(t)| \le C_T .
		\end{equation*}
		\item[(ii)] For each $k\in[d]$, the extremal velocities $
		M^{(k)}(t)$ and
		$m^{(k)}(t)
		$
		are Lipschitz continuous on $[0,T]$ with a uniform Lipschitz constant $C_T$.
		
		\item[(iii)] For each $k\in[d]$, it holds that $M^{(k)}(t)$ is non-increasing, $m^{(k)}(t)$ is non-decreasing and the componentwise velocity diameter
		$D^{(k)}(t)$
		is non-increasing in time on $[0,T]$. 
	\end{itemize}
\end{lemma}
\begin{proof} (i). Since $(x,v)\in C([0,T];\mathcal B)$ and $[0,T]$ is compact, there exists $K_T<\infty$ such that
\[
\sup_{t\in[0,T]}\|v(t)\|_\infty\le K_T.
\]
Since $|\Gamma^{(k)}(s)|=|s|^\alpha$ and $0\le H\le\|H\|_\infty$, we have, uniformly in $i$ and $t\in(0,T)$,
\begin{equation*}
|\dot v_i^{(k)}(t)|\le\|H\|_\infty\sum_{j=1}^\infty m_j|v_j^{(k)}(t)-v_i^{(k)}(t)|^\alpha.
\end{equation*}
Since $|v_i^{(k)}(t)|\le K_T$ for every $i$, we obtain
\begin{equation}\begin{aligned}\label{B-1}
|\dot v_i^{(k)}(t)|&\le M\|H\|_\infty(2K_T)^\alpha=:C_T.
\end{aligned}\end{equation}

\noindent(ii). From \eqref{B-1}, it follows that 
for any $s,t\in[0,T]$ and any $i\in\mathbb{N}$,
\[
|v_i^{(k)}(t)-v_i^{(k)}(s)|\le C_T|t-s|.
\]
We take supremum over $i$ to get
\[
\|v^{(k)}(t)-v^{(k)}(s)\|_\infty \le C_T|t-s|,\quad \forall \ s,t\in[0,T].
\]
By definition of $M^{(k)}(t)$, we have
\[
\begin{aligned}
	M^{(k)}(t)
	&= \sup_{i\ge1} v_i^{(k)}(t)
	\le \sup_{i\ge1}\bigl( v_i^{(k)}(s) + |v_i^{(k)}(t)-v_i^{(k)}(s)| \bigr)\\
	& \le \sup_{i\ge1} v_i^{(k)}(s) + \sup_{i\ge1}|v_i^{(k)}(t)-v_i^{(k)}(s)|\\
	&= M^{(k)}(s) + \|v^{(k)}(t)-v^{(k)}(s)\|_\infty .
\end{aligned}
\]
We interchange the roles of $t$ and $s$, then
\[
M^{(k)}(s) \le M^{(k)}(t) + \|v^{(k)}(t)-v^{(k)}(s)\|_\infty ,
\]
which implies for any $s,t\in[0,T]$,
\[
|M^{(k)}(t)-M^{(k)}(s)|\le \|v^{(k)}(t)-v^{(k)}(s)\|_\infty\le C_T|t-s|
.
\]
Analogously, it follows that for any $s,t\in[0,T]$,
\[|m^{(k)}(t)-m^{(k)}(s)|\le \|v^{(k)}(t)-v^{(k)}(s)\|_\infty\le C_T|t-s|.\]
Thus $M^{(k)}$ and $m^{(k)}$ are Lipschitz on $[0,T]$ with Lipschitz coefficient $C_T$.

\vspace{0.2cm}

\noindent(iii). From (ii), the right Dini derivative
\[
\mathcal{D}^+M^{(k)}(t):=\limsup_{h\to0^+}\frac{M^{(k)}(t+h)-M^{(k)}(t)}{h}
\]
is well-defined for all $t\in[0,T)$. It is well-known that if $\mathcal{D}^{+}f(t)\le0$ for all $t$, then $f$ is non-increasing. Fix $t\in[0,T)$ and take $h>0$ such that $t+h\le T$. Choose an index $i_h$ such that
\begin{equation}\label{B-2}
	v_{i_h}^{(k)}(t+h)\ge M^{(k)}(t+h)-h^2.
\end{equation}
Then
\begin{align*}
	\frac{M^{(k)}(t+h)-M^{(k)}(t)}{h}
	&\le \frac{v_{i_h}^{(k)}(t+h)+h^2-v_{i_h}^{(k)}(t)}{h}=\frac{v_{i_h}^{(k)}(t+h)-v_{i_h}^{(k)}(t)}{h}+h.
\end{align*}
By the mean value theorem, there exists $\xi_h\in(t,t+h)$ such that
\[
\frac{v_{i_h}^{(k)}(t+h)-v_{i_h}^{(k)}(t)}{h}=\dot v_{i_h}^{(k)}(\xi_h).
\]
Hence, we have
\begin{equation}\label{B-3}
	\frac{M^{(k)}(t+h)-M^{(k)}(t)}{h}\le \dot v_{i_h}^{(k)}(\xi_h)+h.
\end{equation}
We now estimate $\dot v_{i_h}^{(k)}(\xi_h)$. We first recall \eqref{A-1} and get
\[
\dot v_{i_h}^{(k)}(\xi_h)=\sum_{j=1}^{\infty} m_jH_{i_hj}(\xi_h)
\Gamma^{(k)}\big(v_j^{(k)}(\xi_h)-v_{i_h}^{(k)}(\xi_h)\big).
\]
Since $v_j^{(k)}(\xi_h)\le M^{(k)}(\xi_h)$ for all $j$, and $\Gamma^{(k)}$ is nondecreasing on $\mathbb R$, we have for any $j\in\mathbb{N}$,
\[
\Gamma^{(k)}\big(v_j^{(k)}(\xi_h)-v_{i_h}^{(k)}(\xi_h)\big)
\le \Gamma^{(k)}\big(M^{(k)}(\xi_h)-v_{i_h}^{(k)}(\xi_h)\big).
\]
Because $0\le H\le\|H\|_\infty$ and $\sum_j m_j=M$, we have
\begin{equation}\label{B-4}
	\dot v_{i_h}^{(k)}(\xi_h)
	\le M\|H\|_\infty\,\Gamma^{(k)}\big(M^{(k)}(\xi_h)-v_{i_h}^{(k)}(\xi_h)\big).
\end{equation}
Next, from the Lipschitz continuity of $M^{(k)}$ and $v_{i_h}^{(k)}$ together with \eqref{B-2}, one has
\[
\begin{aligned}
	0&\le M^{(k)}(\xi_h)-v_{i_h}^{(k)}(\xi_h)\\
	&\le |M^{(k)}(\xi_h)-M^{(k)}(t+h)|
	+\! |M^{(k)}\!(t+h)\!-v_{i_h}^{(k)}\!(t+h)|\!
	+ \! |v_{i_h}^{(k)}(t+h)-v_{i_h}^{(k)}(\xi_h)|
	\\
	&\le C_T h + h^2 + C_T h\to 0, \quad \text{as} \quad h\to0^+.
\end{aligned}
\]
Further, by the definition of $\Gamma^{(k)}$, we have
\[
\lim_{h\to 0^+}\Gamma^{(k)}\big(M^{(k)}(\xi_h)-v_{i_h}^{(k)}(\xi_h)\big)=0.
\]
We combine this with \eqref{B-3}-\eqref{B-4} to obtain
\[
\mathcal{D}^+M^{(k)}(t)\le 0,\quad\forall\ t\in(0,T).
\]
Analogously, we have 
\[
\mathcal{D}^-m^{(k)}(t):=\liminf_{h\to0^+}\frac{m^{(k)}(t+h)-m^{(k)}(t)}{h}\ge 0,\ \forall\ t\in(0,T).
\]
Indeed, the argument is analogous. Fix $t$ and choose $j_h$ such that
\[
v_{j_h}^{(k)}(t+h)\le m^{(k)}(t+h)+h^2.
\]
Then, we have
\begin{align*}
	\frac{m^{(k)}(t+h)-m^{(k)}(t)}{h}
	&\ge \frac{v_{j_h}^{(k)}(t+h)-v_{j_h}^{(k)}(t)}{h}-h= \dot v_{j_h}^{(k)}(\eta_h)-h
\end{align*}
for some $\eta_h\in(t,t+h)$. Since $v_j^{(k)}(\eta_h)\ge m^{(k)}(\eta_h)$ for all $j$ and $\Gamma^{(k)}$ is nondecreasing,
\[
\Gamma^{(k)}\big(v_j^{(k)}(\eta_h)-v_{j_h}^{(k)}(\eta_h)\big)
\ge \Gamma^{(k)}\big(m^{(k)}(\eta_h)-v_{j_h}^{(k)}(\eta_h)\big),
\]
and with $0\le H\le\|H\|_\infty$ and $\sum_j m_j=M$, we obtain
\[
\dot v_{j_h}^{(k)}(\eta_h)\ge M\|H\|_\infty\,\Gamma^{(k)}\big(m^{(k)}(\eta_h)-v_{j_h}^{(k)}(\eta_h)\big).
\]
Moreover, since 
\[
\begin{aligned}
	0&\le v_{j_h}^{(k)}(\eta_h)-m^{(k)}(\eta_h)\\
	&\le |m^{(k)}(\eta_h)-m^{(k)}(t+h)|+ \!|m^{(k)}(t+h)\!-\! v_{j_h}^{(k)}(t+h)|
	+ |v_{j_h}^{(k)}\!(t+h)\!-\! v_{j_h}^{(k)}\!(\eta_h)|
	\\
	&\le C_T h + h^2 + C_T h\to 0, \quad \text{as} \quad h\to0^+,
\end{aligned}
\]
we have
\[
\lim_{h\to 0^+}\Gamma^{(k)}\big(m^{(k)}(\eta_h)-v_{j_h}^{(k)}(\eta_h)\big)=0.
\]
Hence, 
\[ \mathcal{D}^-m^{(k)}(t)\ge0,\quad \forall\ t\in(0,T).\]
Finally, since $D^{(k)}(t)=M^{(k)}(t)-m^{(k)}(t)$, we have 
\[
\mathcal{D}^+D^{(k)}(t)\le \mathcal{D}^+M^{(k)}(t)-\mathcal{D}^-m^{(k)}(t)\le 0,\quad \forall\ t\in(0,T)
\]
so $D^{(k)}(\cdot)$ is non-increasing on $[0,T]$.
\end{proof}

\vspace{0.3cm}

\begin{remark}\label{R:2.1}
	\begin{enumerate}
		\item From the monotonicity properties of the velocity components $M^{(k)}$ and $m^{(k)}$ in Lemma \ref{L2.1}, we have a rough estimate of $\|v\|_\infty$:
		\begin{equation*}
			\|v(t)\|_\infty\le\left(\sum_{k=1}^d \max\big\{|M^{(k)}(0)|,|m^{(k)}(0)|\big\}^2\right)^{1/2},\quad\forall \ 
			t\ge 0.
		\end{equation*}
		Due to the natural relation between $x$ and $v$ in \eqref{A-1}, from Lemma \ref{L2.1} we have $\forall \ 
		t\ge 0$,
		\begin{equation}\label{B-5}
			\begin{aligned}
				&\sup_{i,j}\|x_{i}(t)-x_j(t)\|\le\sup_{i,j}\|x_{i}^{\rm in}-x_j^{\rm in}\|+\left(\sum_{k=1}^d (D^{(k)}(0))^2\right)^{1/2}t=:a+bt.	\end{aligned}\end{equation}
		\item From \eqref{B-5}, we have a rough nonnegative lower bound of communication weight
		\[
		\inf_{i,j}H_{ij}(t)\ge H(a+bt),\quad\forall \ 
		t\ge0.
		\]
	\end{enumerate}
\end{remark}

\vspace{0.2cm}
Let $v_c(t)$ denote the mass-weighted mean velocity:
\begin{equation}\label{B-6}
	v_c(t):=\frac{1}{M}\sum_{i=1}^{\infty}m_i v_i(t),\quad v_c^{\rm in}:=\frac{1}{M}\sum_{i=1}^{\infty}m_i v_i^{\rm in}.
\end{equation}
We show below that $v_c(t)$ is constant in time.
\begin{lemma}\label{L2.2}
	Let $(x,v)\in C([0,\infty);\mathcal B)\cap C^1((0,\infty);\tilde{\mathcal H})$
	be a solution to system \eqref{A-1} with initial data
	$(x^{\rm in},v^{\rm in})\in \mathcal{B}$. Then $v_c(t)$ is conserved:
\[
\frac{\di}{\di t}v_c(t)=0,\quad v_c(t)\equiv v_c^{\rm in}.
\]
\end{lemma}
\begin{proof} On any interval $[0,T]$, continuity in $\mathcal B$ yields $\sup_{t\in[0,T]}\|v(t)\|_\infty\le K_T<\infty$. Hence
\[
\sum_{i,j\ge1}m_im_j\|H_{ij}(t)\Gamma(v_j(t)-v_i(t))\|
\le M^2\|H\|_\infty d^{(1-\alpha)/2}(2K_T)^\alpha<\infty
\]
uniformly for $t\in[0,T]$. In particular, the series defining $v_c$ and its derivative converge locally uniformly, so termwise differentiation and rearrangement of the absolutely convergent double sum are justified. We differentiate \eqref{B-6} and use \eqref{A-1} to obtain:
\begin{equation}\label{B-7}
	\begin{aligned}
		\frac{\di}{\di t}v_c(t)
		&=\frac{1}{M}\sum_{i=1}^{\infty} m_i\,\dot v_i(t)=\frac{1}{M}\sum_{i=1}^{\infty}\sum_{j=1}^{\infty} m_i m_j\,H_{ij}(t)\,\Gamma\big(v_j(t)-v_i(t)\big).
	\end{aligned}
\end{equation}
Since $H_{ij}(t)=H_{ji}(t)$, we get
\begin{align*}
	\sum_{i=1}^{\infty}\sum_{j=1}^{\infty} m_i m_j H_{ij}\Gamma(v_j-v_i)=\sum_{i=1}^{\infty}\sum_{j=1}^{\infty} m_i m_j H_{ij}\Gamma(v_i-v_j).
\end{align*}
Then it follows from the above and $\Gamma(v)=-\Gamma(-v)$ that
\begin{align*}
	&2\sum_{i=1}^{\infty}\sum_{j=1}^{\infty} m_i m_j H_{ij}\Gamma(v_j-v_i)=\sum_{i=1}^{\infty}\sum_{j=1}^{\infty} m_i m_j H_{ij}\Big(\Gamma(v_j-v_i)+\Gamma(v_i-v_j)\Big)=0.
\end{align*}
Therefore, we go back to \eqref{B-7} to obtain that $$\frac{\di}{\di t}v_c(t)=0,$$ i.e. $v_c(t)$ is constant in time. 
\end{proof}

Finally, we recall a sublinear differential inequality for Dini derivatives, which will be used to obtain finite-time decay estimates.

\begin{lemma}\label{L2.3}
	Let $\alpha\in(0,1)$ and let $k\in L^1_{\mathrm{loc}}([0,\infty))$ satisfy $k(t)\ge 0$ a.e.
	Assume that $y:[0,\infty)\to[0,\infty)$ is locally absolutely continuous and satisfies
	\begin{equation}\label{B-8}
		\mathcal{D}^+ y(t)\le -k(t)\,y(t)^\alpha,
		\quad\text{for a.e. }t\ge 0.
	\end{equation}
	If for some time $T^*\le\infty$, 
	\[
	\int_0^{T^*} k(s)\,\di s = \frac{y(0)^{1-\alpha}}{1-\alpha}.
	\]
	Then, it holds that 
	\begin{equation}\label{B-9}
		y(t)\le \Bigl(y(0)^{1-\alpha}-(1-\alpha)\int_0^t k(s)\,\di s\Bigr)^{\frac1{1-\alpha}},\quad \forall~ t< T^*.
	\end{equation}
	Moreover, if $T^*<\infty$, then $y(t)=0$ for all $t\ge T^*$.
\end{lemma}

\begin{proof} Since $k\ge0$ and $y$ is locally absolutely continuous, \eqref{B-8} implies $y'(t)\le0$ for almost every $t\ge0$. Thus $y$ is non-increasing, and if $y(0)=0$, then $y(t)=0$ for all $t\ge0$.
Suppose $y(0)>0$ and let $T_0:=\inf\{t\ge0:y(t)=0\}$, with $T_0=\infty$ if $y$ never vanishes. For every $t<T_0$, the function
\[
z(t):=y(t)^{1-\alpha}
\]
is absolutely continuous on $[0,t]$. The chain rule and \eqref{B-8} yield for almost every $s\in[0,t]$,
\[
z'(s)=(1-\alpha)y(s)^{-\alpha}y'(s)\le-(1-\alpha)k(s).
\]
Integration from $0$ to $t$ gives
\begin{equation*}
y(t)^{1-\alpha}\le y(0)^{1-\alpha}-(1-\alpha)\int_0^t k(s)\,\di s.
\end{equation*}
This proves \eqref{B-9} for $t<\min\{T_0,T^*\}$. If $T_0<T^*$, the same bound is immediate for $T_0\le t<T^*$ because $y$ remains zero. If $T^*<\infty$ and $T_0>T^*$, letting $t\uparrow T^*$ in the preceding estimate contradicts $y(T^*)>0$. Hence $T_0\le T^*$, and monotonicity yields $y(t)=0$ for every $t\ge T^*$.
\end{proof}

\section{Finite-time flocking} \label{sec:3}
\setcounter{equation}{0}
In this section, we study finite-time flocking under fixed sender networks and then extend the analysis to switching sender networks.

\subsection{Componentwise diameter estimates}
We first present right Dini derivative estimates for $M^{(k)}(t)$ and $m^{(k)}(t)$.

\begin{lemma}\label{L3.1}
	Let $0<\alpha<1$. Let $(x,v)\in C([0,\infty);\mathcal B)\cap C^1((0,\infty);\tilde{\mathcal H})$ be a solution to system \eqref{A-1} with initial data $(x^{\rm in},v^{\rm in})\in \mathcal{B}$. Assume that there exists certain nonnegative continuous function $H_*:[0,\infty)\to[0,\infty)$ such that
	\[
	\inf_{i,j\in \mathbb{N}}H_{ij}(t)\ge H_*(t),\quad  \forall~ t\ge0.
	\]
	Then, for each $k\in[d]$ and $t>0$, the following estimates hold:
	\begin{align*}
		\mathcal{D}^+M^{(k)}(t)
		&\le
		-H_*(t)\sum_{j=1}^{\infty} m_j\,\big|M^{(k)}(t)-v_j^{(k)}(t)\big|^\alpha,
		\\[2mm]
		\mathcal{D}^-m^{(k)}(t)
		&\ge
		H_*(t)\sum_{j=1}^{\infty} m_j\,\big|v_j^{(k)}(t)-m^{(k)}(t)\big|^\alpha.
	\end{align*}
	
\end{lemma}
\begin{proof} Recall \eqref{B-3} and consider the dynamics of $ v_{i_h}^{(k)}(\xi_h)$:
\begin{equation}\label{C-1}
	\begin{aligned}
		\dot v_{i_h}^{(k)}(\xi_h)
		&=
		\sum_{j=1}^{\infty} m_j\,H_{i_hj}(\xi_h)\,
		\Gamma^{(k)}\big(v_j^{(k)}(\xi_h)-v_{i_h}^{(k)}(\xi_h)\big)\\
		&=-\sum_{j=1}^{\infty} m_j\,H_{i_hj}(\xi_h)|v_j^{(k)}(t)-M^{(k)}(t)|^\alpha\\
		&\hspace{0.2cm}+
		\sum_{j=1}^{\infty} m_j\,H_{i_hj}(\xi_h)\,
		\Gamma^{(k)}\big(v_j^{(k)}(\xi_h)-v_{i_h}^{(k)}(\xi_h)\big)\\
		&\hspace{0.2cm}-\sum_{j=1}^{\infty} m_j\,H_{i_hj}(\xi_h)\Gamma^{(k)}\big(v_j^{(k)}(t)-M^{(k)}(t)\big)\\
		&\le\!\!
		-\! H_*(\xi_h)\!\sum_{j=1}^{\infty}\! m_j|v_j^{(k)}\!(t)\!-\!\! M^{(k)}\!(t)|^\alpha\!\!+\!\sum_{j=1}^{\infty}\!\! m_j\mathcal{I}_j(h),
\end{aligned}\end{equation}
where 
\begin{align*}
	\mathcal{I}_j(h):&=H_{i_hj}(\xi_h)\,
	\left|\Gamma^{(k)}\big(v_j^{(k)}(\xi_h)-v_{i_h}^{(k)}(\xi_h)\big)-\Gamma^{(k)}\big(v_j^{(k)}(t)-M^{(k)}(t)\big)\right|.
\end{align*}
Fix $T>t$ and take $h>0$ sufficiently small.
By \eqref{B-2} and the uniform Lipschitz bounds in
Lemma~\ref{L2.1}, it holds
\[
\begin{aligned}
|v_{i_h}^{(k)}(\xi_h)-M^{(k)}(t)|
&\le
|v_{i_h}^{(k)}(\xi_h)-v_{i_h}^{(k)}(t+h)|\\
&\hspace{0.4cm}
+|v_{i_h}^{(k)}(t+h)-M^{(k)}(t+h)|
+|M^{(k)}(t+h)-M^{(k)}(t)|\\
&\le 2C_T h+h^2.
\end{aligned}
\]
Hence, we have that for every $j\in\mathbb N$,
\[
\left|
\bigl(v_j^{(k)}(\xi_h)-v_{i_h}^{(k)}(\xi_h)\bigr)
-\bigl(v_j^{(k)}(t)-M^{(k)}(t)\bigr)
\right|
\le 3C_T h+h^2.
\]
The scalar sign-power function satisfies
\[
|\Gamma^{(k)}(r)-\Gamma^{(k)}(s)|
\le 2^{1-\alpha}|r-s|^\alpha,
\quad r,s\in\mathbb R.
\]
It follows that
\[
0\le \sum_{j=1}^{\infty}m_j\mathcal I_j(h)
\le
2^{1-\alpha}M\|H\|_\infty
(3C_T h+h^2)^\alpha
\longrightarrow 0.
\]
Since $H_*$ is continuous, we go back to \eqref{C-1} to obtain
\[
\limsup_{h\to0^+}\dot v_{i_h}^{(k)}(\xi_h)\le -H_*(t)\sum_{j=1}^{\infty} m_j|v_j^{(k)}(t)- M^{(k)}(t)|^\alpha.\]
Hence, with \eqref{B-3} we have  
\[
\mathcal{D}^+M^{(k)}(t)
\le
-H_*(t)\sum_{j=1}^{\infty} m_j|v_j^{(k)}(t)- M^{(k)}(t)|^\alpha.
\]
Analogously, we have 
\[
\mathcal{D}^-m^{(k)}(t)
\ge
H_*(t)\sum_{j=1}^{\infty} m_j|v_j^{(k)}(t)- m^{(k)}(t)|^\alpha.
\]
\end{proof}

\subsection{Finite-time flocking under fixed sender networks}
We are now ready to state our main result.

\begin{theorem}\label{T3.1}
	Let $0<\alpha<1$.
	Let $(x,v)\in C([0,\infty);\mathcal B)\cap C^1((0,\infty);\tilde{\mathcal H})$
	be a solution to system \eqref{A-1} with initial data
	$(x^{\rm in},v^{\rm in})\in \mathcal{B}$.
	If, for every $k\in[d]$ with $D^{(k)}(0)>0$, there exists a time $T_*^{(k)}<\infty$ such that
	\begin{equation}\label{C-2}
		\int_0^{T_*^{(k)}} H(a+bs)\,\di s = \frac{D^{(k)}(0)^{1-\alpha}}{M(1-\alpha)},
	\end{equation}
	where $a,b$ are defined in Remark \ref{R:2.1}, and $T_*^{(k)}:=0$ when $D^{(k)}(0)=0$,
	then for each $k\in[d]$, for $0\le t\le T_*^{(k)}$
	\begin{equation*}
		D^{(k)}\!(t)\!\le \Bigl((D^{(k)}(0))^{1-\alpha}-M(1-\alpha)\int_0^t H(a+bs)\,\di s\Bigr)^{\frac1{1-\alpha}},
	\end{equation*}
	and 
	\[
	D^{(k)}(t)=0,\quad \forall~ t\ge T_*^{(k)}.
	\]
	Moreover, let $T_*:=\max_{k\in[d]}\inf\{t\ge0: D^{(k)}(t)=0\}\le\max_{k\in[d]}T_*^{(k)}$, then finite-time flocking emerges:
	\begin{itemize}
		\item[(i)] Finite-time velocity alignment: \begin{equation*}
			v_i(t)=v_c^{\rm in},\quad \forall~ i,\ \forall~ t\ge T_*,
		\end{equation*}
		in particular,
		\[
	\sum_{i=1}^{\infty}m_i\|v_i(t)-v_c^{\rm in}\|^2=0,\quad \forall~t\ge T_*.
	\]
		\item[(ii)] Freezing of relative positions:
		\begin{equation*}
			\sup_{t\ge0}\sup_{i,j}\|x_i(t)-x_j(t)\|\le a+bT_*.
		\end{equation*}
		Moreover, let $\Phi(t):=\int_0^t H(s)\,\di s$. If
		\[
		\Phi(D_x(0))+\frac{1}{M(2-\alpha)}\sum_{k=1}^d(D^{(k)}(0))^{2-\alpha}<\int_0^\infty H(s)\,\di s,
		\]
		then the sharper bound
		\[
		\sup_{t\ge0}D_x(t)\le \Phi^{-1}\!\left(\Phi(D_x(0))+\frac{1}{M(2-\alpha)}\sum_{k=1}^d(D^{(k)}(0))^{2-\alpha}\right)=:x^*
		\]
		holds, where $\Phi^{-1}$ is taken on the strictly increasing part of $\Phi$.
		\item[(iii)] Under the additional condition in (ii), the flocking time $T_*$ satisfies
		\begin{equation}\label{C-3}
			T_*\le \frac{\displaystyle\max_{k\in[d]}(D^{(k)}(0))^{1-\alpha}}{M(1-\alpha)H(x^*)}.
		\end{equation}	
	\end{itemize}
\end{theorem}

\begin{proof} \textbf{Step A. } We claim that if there exists certain nonnegative continuous function $H_*:[0,\infty)\to[0,\infty)$ such that
\[
\inf_{i,j\in \mathbb{N}}H_{ij}(t)\ge H_*(t),\quad  \forall~ t\ge0.
\]
then for almost every $t>0$, the velocity diameter $D^{(k)}(t)$ satisfies
\begin{equation}\label{C-4}
	\mathcal{D}^+D^{(k)}(t)\le -MH_*(t)\,D^{(k)}(t)^\alpha.
\end{equation}
Since $D^{(k)}(t)=M^{(k)}(t)-m^{(k)}(t)$, the subadditivity of the upper Dini derivative yields
\[
\mathcal{D}^+D^{(k)}(t)\le \mathcal{D}^+M^{(k)}(t)-\mathcal{D}^-m^{(k)}(t).
\]
Applying Lemma~\ref{L3.1}, we obtain
\begin{align*}
	\mathcal{D}^+D^{(k)}(t)&\le
	-H_*(t)\sum_{j=1}^{\infty} m_j
	\Big[|M^{(k)}(t)-v^{(k)}_j(t)|^\alpha+|v^{(k)}_j(t)-m^{(k)}(t)|^\alpha\Big]\\
	&\le-H_*(t)\sum_{j=1}^{\infty} m_j\big(M^{(k)}(t)-m^{(k)}(t)\big)^\alpha=-MH_*(t)(D^{(k)})^{\alpha},
\end{align*}
where we use the fact that for $	\alpha\in(0,1),$
\[
a^\alpha+b^\alpha\ge(a+b)^\alpha,\quad \forall~ a,\ b\ge0,
\]
and 
\[M^{(k)}(t)-v^{(k)}_j(t)\ge0,\quad v^{(k)}_j(t)-m^{(k)}(t)\ge0.\]
This proves \eqref{C-4}.

\vspace{0.2cm}

\noindent \textbf{Step B.} By Remark \ref{R:2.1}, we have that for all $t\ge0$,
\[
\inf_{i,j}H_{ij}(t)\ge H(a+bt)=:H_*(t).
\]
For each fixed $k\in[d]$, Lemma~\ref{L2.1} implies that
$y(t):=D^{(k)}(t)$ is locally Lipschitz and hence locally
absolutely continuous. Moreover, \eqref{C-4} gives
\[
\mathcal D^+ D^{(k)}(t)\le -MH_*(t)\,(D^{(k)}(t))^\alpha \quad \text{for a.e. }t\ge0.
\]
By \eqref{C-2}, for $D^{(k)}(0)>0$ there exists a time $T_*^{(k)}<\infty$ such that 
\begin{equation*}
	\int_0^{T_*^{(k)}}H_*(s)\ \di s = \frac{(D^{(k)}(0))^{1-\alpha}}{M(1-\alpha)},
\end{equation*}
Therefore, Lemma~\ref{L2.3} yields that $\forall\  t< T_*^{(k)},$	
\begin{equation*}
	D^{(k)}\!(t)\!\le \Bigl((D^{(k)}(0))^{1-\alpha}-M(1-\alpha)\int_0^t H(a+bs)\,\di s\Bigr)^{\frac1{1-\alpha}},
\end{equation*}
and $D^{(k)}(t)=0$ for all $t\ge T_*^{(k)}$; if $D^{(k)}(0)=0$, the same assertion follows from Lemma~\ref{L2.1}.

\vspace{0.2cm}

\noindent\textbf{Step C.} For (i), we have that for $t\ge T_*$,  
\[v_i(t)=v_j(t),\quad \forall~ i,j.
\]
Since the weighted average velocity is conserved in Lemma \ref{L2.2}, we have 
\[
v_c^{\rm in}=\frac{1}{M}\sum_j m_j v_j(t)=v_i(t),\quad \forall~ i, \ \forall~ t\ge T_*.
\]
Next, we prove (ii). Denote
\[
D_x(t):=\sup_{i,j}\|x_i(t)-x_j(t)\|.
\]
By \eqref{B-5}, we have $D_x(t)\le a+bt$ for $0\le t\le T_*$.
The relative positions remain constant after $T_*$.
Hence
\[
\sup_{t\ge0}D_x(t)\le a+bT_*.
\]
The identity $\dot x_i=v_i$ and the uniform velocity bound imply
that each function $t\mapsto\|x_i(t)-x_j(t)\|$ is Lipschitz
continuous with a constant independent of $i$ and $j$.
Their supremum $D_x(t)$ is therefore also Lipschitz continuous.
For $h>0$, the identity $\dot x_i=v_i$ gives
\[
\begin{aligned}
D_x(t+h)
&\le D_x(t)
+\int_t^{t+h}\sup_{i,j}\|v_i(s)-v_j(s)\|\,\di s\le D_x(t)
+\int_t^{t+h}\sum_{k=1}^d D^{(k)}(s)\,\di s.
\end{aligned}
\]
Since each $D^{(k)}$ is continuous, division by $h$ and
passage to the upper limit yield
\[
\mathcal D^+D_x(t)\le\sum_{k=1}^dD^{(k)}(t).
\]
Consequently, by the chain rule for Dini derivatives,
\begin{equation}\label{C-6}
	\begin{aligned}
		\mathcal{D}^+\Phi(D_x(t))
		&	=H(D_x(t))\,\mathcal{D}^+D_x(t)\le H(D_x(t))\sum_{k=1}^d D^{(k)}(t).
	\end{aligned}
\end{equation}
On the other hand, for each $k\in[d]$, we choose $H_*(t)=H(D_x(t))$, by the dissipation estimate for $D^{(k)}$, we have
\[
\mathcal{D}^+D^{(k)}(t)
\le -\,M\,H(D_x(t))\,\big(D^{(k)}(t)\big)^\alpha,
\]
which implies
\begin{align*}
	\mathcal{D}^+(D^{(k)}(t))^{2-\alpha}
	&=(2-\alpha)(D^{(k)}(t))^{1-\alpha}\,\mathcal{D}^+D^{(k)}(t)\le -M(2-\alpha)\,H(D_x(t))\,D^{(k)}(t).
\end{align*}
Summing over $k=1,\dots,d$, we obtain
\begin{equation}\label{C-7}
	\sum_{k=1}^d \mathcal{D}^+(D^{(k)}(t))^{2-\alpha}
	\le -M(2-\alpha)\,H(D_x(t))\sum_{k=1}^d D^{(k)}(t).
\end{equation}
Define the Lyapunov-type functional
\[
\mathcal L(t)
:=\Phi(D_x(t))
+\frac{1}{M(2-\alpha)}
\sum_{k=1}^d (D^{(k)}(t))^{2-\alpha}.
\]
We combine \eqref{C-6} and \eqref{C-7} to obtain
\begin{align*}
	\mathcal{D}^+\mathcal L(t)
	&\le \mathcal{D}^+\Phi(D_x(t))\!
	+\!\frac{1}{M(2-\alpha)}\sum_{k=1}^d \mathcal{D}^+\!(D^{(k)}(t))^{2-\alpha}\\
	&\le H(D_x(t))\sum_{k=1}^d D^{(k)}(t)\!
	-\!H(D_x(t))\sum_{k=1}^d D^{(k)}(t)=0.
\end{align*}
Hence, $\mathcal L(t)$ is non-increasing on $[0,\infty)$.
This gives
\[
\Phi(D_x(t))
\le
\Phi(D_x(0))
+\frac{1}{M(2-\alpha)}\sum_{k=1}^d (D^{(k)}(0))^{2-\alpha},
\]
for all $0\le t\le T_*$. Since $H$ is continuous, nonnegative, and non-increasing,
if $H$ vanishes at a finite point, it remains zero thereafter.
Consequently, $\Phi$ is strictly increasing on the interval
where $H$ is positive. The strict inequality in (ii) places
the displayed argument below $\int_0^\infty H(s)\,\di s$.
Therefore, $\Phi^{-1}$ is well defined at that argument,
$x^*$ is finite, and $H(x^*)>0$. Hence we have 
\[
D_x(t)
\le
\Phi^{-1}\left(\Phi(D_x(0))
+\frac{1}{M(2-\alpha)}\sum_{k=1}^d (D^{(k)}(0))^{2-\alpha}\right).
\]
For (iii), under the same additional condition, $H(D_x(t))\ge H(x^*)>0$. Applying Lemma~\ref{L2.3} with $k(t)=MH(x^*)$ to each component gives \eqref{C-3}.
\end{proof}

\vspace{0.3cm}

\begin{remark}\begin{enumerate}
		\item[(i)] If the communication function is non-integrable, i.e.,
		\begin{equation}\label{C-8}
			\int_0^{\infty}H(s)\,\di s=+\infty,
		\end{equation} 
		we have unconditional finite-time flocking: for any initial data in \(\mathcal B\) and each $k\in[d]$, there exists a finite time $T_*^{(k)}\ge0$ such that
		\[
		D^{(k)}(t)=0,\quad \forall~ t\ge T_*^{(k)},
		\]
		Here, one may choose
		\[
		T_*^{(k)}\!\!:=\inf\left\{t>0:\!\!\! \int_0^{t} H(a+bs)\,\di s
		\ge \frac{D^{(k)}(0)^{1-\alpha}}{M(1-\alpha)}\right\}.
		\]
		In particular, for Cucker-Smale communication weight function 
		\[
		H(r)=\frac{1}{(1+r^2)^{\beta}}, \forall~ \ r\ge0.
		\] 
		If decay exponent $\beta\in[0,\frac{1}{2}]$, we have unconditional finite-time flocking of system \eqref{A-1}, while we have conditional finite-time  flocking with $\beta>\frac{1}{2}.$
		
		\vspace{0.2cm}
        
		\item[(ii)] For fixed initial data and communication kernel $H$,
		the earliest time at which the integral in \eqref{C-2}
		reaches the prescribed threshold is non-increasing as
		the total mass $M$ increases, whenever that time is finite.
		This concerns the theoretical alignment-time bound;
		an exact inverse proportionality to $M$ is not asserted
		for a general communication kernel.
		
		\vspace{0.2cm}
        
		\item[(iii)] The componentwise diameter approach also yields finite-time flocking for finite-$N$ particle systems.
		In particular, the obtained finite-time flocking estimates do not explicitly depend on $N$. The same diameter argument also applies to finite truncations with their respective total masses. Convergence of the truncated solutions to a specified infinite-particle solution is not asserted here.
	\end{enumerate}
\end{remark}

\vspace{0.5cm}

\subsection{Extension to switching sender networks}\label{sec:4}
As an extension of the fixed-network analysis, we consider the IPCS dynamics under switching sender networks.
Let $\{\mathcal{G}_p\}_{p\in\mathcal{P}}$ be a finite family of sender networks with total mass $M_p$.
A switching signal $\sigma:[0,\infty)\to\mathcal{P}$ is piecewise constant with switching instants
$$0=t_0<t_1<t_2<\cdots,\quad t_n\to\infty$$ and $\sigma(t)=\sigma_n$ on $[t_n,t_{n+1})$.
For each $p\in\mathcal{P}$, let $m_j^{(p)}$ be the sender weight of $j$-th particle corresponding to sender network graph $\mathcal{G}_p$, and assume
\[
m_j^{(p)}\ge0\quad(j\ge1),\quad 0<M_p:=\sum_{j=1}^{\infty}m_j^{(p)}<\infty\quad(p\in\mathcal P).
\]
The IPCS under switching sender networks reads
\begin{equation}\label{D-1}
	\left\{
	\begin{aligned}
		\dot x_i(t)\!&= v_i(t),\quad i\in\mathbb{N}, \quad t>0,\\
		\dot v_i(t) \!&= \sum_{j=1}^\infty m_j^{(\sigma(t))}\,
		H\big(\|x_j(t)-x_i(t)\|\big)\,
		\Gamma\!\big(v_j(t)-v_i(t)\big).
	\end{aligned}
	\right.
\end{equation}

On each time interval $[t_n,t_{n+1})$, the switching signal $\sigma(t)$ is constant.
Hence, system \eqref{D-1} reduces to the IPCS dynamics
associated with a fixed sender network $\mathcal G_{\sigma_n}$. Since $\mathcal P$ is finite, set
\[
\widehat m_j:=2^{-j}+\sum_{p\in\mathcal P}m_j^{(p)}>0,\quad
\widehat{\mathcal H}:=\ell_{\widehat m}^2(\mathbb R^d)\times\ell_{\widehat m}^2(\mathbb R^d).
\]
The total weight $\sum_j\widehat m_j=1+\sum_pM_p$ is finite and $m_j^{(p)}\le\widehat m_j$ for every mode. Thus the proof of Corollary~\ref{global-existence-model}, with the auxiliary weights $\widehat m$ in place of $m+2^{-j}$, provides a classical solution in
$C([t_n,t_{n+1}];\mathcal B)\cap C^1((t_n,t_{n+1});\widehat{\mathcal H})$ for the fixed mode $\sigma_n$, starting from any datum in $\mathcal B$ at $t_n$. By choosing the value at $t_{n+1}$ as the next initial datum and using $t_n\to\infty$, these solutions can be concatenated into a global continuous, piecewise classical solution. Differentiability at the switching instants is not claimed. 

We now apply the componentwise diameter estimates to obtain a finite-time flocking criterion.

\begin{theorem}\label{T4.1}
	Let $0<\alpha<1$. Let $(x,v)$ be a global piecewise classical solution of \eqref{D-1} with
	$(x^{\rm in},v^{\rm in})\in \mathcal{B}$. If, for every $k\in[d]$ with $D^{(k)}(0)>0$, there exists a time $T_*^{(k)}<\infty$ such that
	\begin{equation}\label{D-2}
		\int_0^{T_*^{(k)}} M_{\sigma(s)}\,H(a+bs)\,\di s
		= \frac{D^{(k)}(0)^{1-\alpha}}{1-\alpha},
	\end{equation}
	where $a,b$ are defined in Remark~\ref{R:2.1} and $T_*^{(k)}:=0$ when $D^{(k)}(0)=0$, then
	\[
	D^{(k)}(t)=0,\quad \forall~ t\ge T_*^{(k)}.
	\]
	In particular, letting $T_*:=\max_{k\in[d]}T_*^{(k)}$, all velocities coincide for $t\ge T_*$, i.e.,
	\[
	v_i(t)=v_j(t),\quad \forall~ i,j,\ \forall~ t\ge T_*,
	\]
	and the relative positions freeze:
	\[
	x_i(t)-x_j(t)=x_i(T_*)-x_j(T_*),\quad \forall~ i,j,\ \forall~ t\ge T_*.
	\]
	Moreover, $\sup_{t\ge0}D_x(t)\le a+bT_*<\infty$.
\end{theorem}
\begin{proof} Fix $n\ge0$. 
Since the proofs of Lemmas~\ref{L2.1} and \ref{L3.1} and Step A of the proof of Theorem~\ref{T3.1} are local in time, we can apply the corresponding estimates on each switching interval. Taking $H_*(t):=H(D_x(t))$, we obtain for each $k\in[d]$ and for a.e. $t\in[t_n,t_{n+1})$,
\begin{equation}\label{D-3}
	\mathcal D^+D^{(k)}(t)
	\le -\,M_{\sigma_n}\,H_*(t)\,\big(D^{(k)}(t)\big)^\alpha\le0.
\end{equation}
By continuity at the switching instants, the componentwise velocity bounds and the monotonicity of $D^{(k)}$ hold globally in time. As in Remark~\ref{R:2.1}, we have
\[
\sup_{i,j}\|x_{i}(t)-x_j(t)\|\!\le  \sup_{i,j}\|x_{i}^{\rm in}-x_j^{\rm in}\|+tb,\quad \forall~ t\ge0.
\]
By the monotonicity of $H$, we have $H_*(t)=H(D_x(t))\ge H(a+bt)$.
Combining this with \eqref{D-3}, and noting that
$M_{\sigma_n}=M_{\sigma(t)}$ on $[t_n,t_{n+1})$, we obtain for each $k\in[d]$ and for a.e. $t\ge0$,
\begin{equation}\label{D-4}
	\mathcal D^+D^{(k)}(t)
	\le -\,M_{\sigma(t)}\,H(a+bt)\,\big(D^{(k)}(t)\big)^\alpha.
\end{equation}
Since the mode family is finite and $t_n\to\infty$,
each compact time interval contains only finitely many switching
instants. The componentwise velocity bounds and the bounded
mode masses yield a uniform bound on the accelerations between
switching instants. Together with continuity at the switching
instants, this implies that each $D^{(k)}$ is locally Lipschitz
and hence locally absolutely continuous.
For each $k$ with $D^{(k)}(0)>0$, by Lemma~\ref{L2.3}
and \eqref{D-2}, we have that for $0\le t<T_*^{(k)}$,
\[
\bigl(D^{(k)}(t)\bigr)^{1-\alpha}
\le \bigl(D^{(k)}(0)\bigr)^{1-\alpha}
-(1-\alpha)\int_0^t M_{\sigma(s)}H(a+bs)\,\di s.
\]
By continuity, $D^{(k)}(T_*^{(k)})=0$.
The monotonicity of $D^{(k)}$ then implies
$D^{(k)}(t)=0$ for all $t\ge T_*^{(k)}$.
If $D^{(k)}(0)=0$, monotonicity gives the same conclusion
with $T_*^{(k)}=0$.
Let $T_*:=\max_{k\in[d]}T_*^{(k)}$. Then we have $D^{(k)}(t)=0$ for all $k\in[d]$ and all $t\ge T_*$. Therefore, we have
\[v_i(t)=v_j(t),\quad \forall~i,j\in\mathbb N, \quad t\ge T_*.\]
Since $\dot x_i=v_i$, we obtain
\[x_i(t)-x_j(t)=x_i(T_*)-x_j(T_*),\quad \forall~i,j\in\mathbb N,\quad t\ge T_*.\]
This, together with \eqref{B-5}, yields
\[\sup_{t\ge0}D_x(t)\le a+bT_*.\]
\end{proof}

\begin{remark}
	The criterion \eqref{D-2} depends on the switching signal
	through the total sender mass $M_{\sigma(t)}$.
	If $M_p=M$ for every mode $p$, it reduces to the fixed-mass
	criterion \eqref{C-2} and is independent of the switching signal.
	Thus, the criterion provides a uniform sufficient condition
	for switching among sender-weight configurations with the
	same total mass. It does not distinguish the distributions
	of individual sender weights with identical total mass;
	the actual trajectories and alignment times may nevertheless
	depend on these distributions.
\end{remark}

\vspace{0.5cm}

\section{Conclusion}\label{sec:6}

In this paper, we studied finite-time flocking for an infinite-particle
Cucker-Smale type system with nonlinear coupling under sender networks.
We established global existence of classical solutions for fixed sender networks, allowing for zero sender weights.
Using a componentwise diameter approach and sublinear dissipation estimates,
we established sufficient conditions for finite-time flocking.
In particular, if the communication weight function is non-integrable,
then finite-time flocking emerges for any uniformly bounded initial data.
The alignment-time estimates apply to both finite and countably
infinite systems and contain no explicit dependence on the
number of particles when the total mass and initial diameter
bounds are controlled. As an extension, we established global existence of piecewise classical solutions and a finite-time flocking criterion for switching sender networks. Future work will address uniqueness
and stability, convergence of finite-particle approximations,
and extensions to more general interaction structures.

\vspace{0.3cm}

\section*{Conflict of interest statement}
The authors declare no conflicts of interest.

\section*{Data availability statement}
Data sharing is not applicable to this article as no datasets were generated or analyzed during the current study.

\section*{Declaration of generative AI use}
During the preparation of this work, the last author used ChatGPT (OpenAI) to assist with language editing and consistency checks. The authors reviewed and verified all content and take full responsibility for the content of the published article.

\section*{Acknowledgments}
The work of S.-Y. Ha, X. Y. Wang and F. Q. Zeng is supported by National Research Foundation(NRF) grant funded by the Korea government(MSIT) (RS-2025-00514472). The work of X. Wang is also supported by the Natural Science Foundation of China (grants  123B2003), the China Postdoctoral Science Foundation (grants 2025M774290), and Heilongjiang Province Postdoctoral Funding (grants LBH-Z24167).

\vspace{0.5cm}

\appendix
\section{Global existence of classical solutions to infinite-particle systems}    
In this appendix, we give the global existence of classical solutions to infinite-particle systems on space  $\mathcal{B}$ by using Schauder-Tychonoff fixed point and a global continuation argument. We consider the Cauchy problem for $u = (x,v)$:
\begin{equation}\label{G-1}
	\begin{cases}
		\dfrac{\di u}{\di t} = F(u), & t>0, \\
		u(0) = u^{\rm in}\in\mathcal{B}.
	\end{cases}
\end{equation}

%\begin{theorem}[\cite{M-N}]\label{thm:Banach-existence}
%	Let $\mathcal{H}$ be a Hilbert space and let $F: \mathcal{S}\subset \mathcal{H}\to \mathcal{S}\subset \mathcal{H}$ satisfy the following conditions:
%	\begin{enumerate}
	%		\item $\mathcal{S}$ is a convex compact set on $\mathcal{H}$;
	%		\item $F$ is  continuous from $\mathcal{S}$ to $\mathcal{S}$ with respect to norm $\|\cdot\|_{\mathcal{H}}$.
	%	\end{enumerate}
%	Then for any initial datum $u^{\mathrm{in}}\in \mathcal{S}$, there exists a time $T>0$ and a function
%	\[
%	u\in C([0,T];\mathcal{H})
%	\]
%	such that $u$ is a solution to the integral equation
%	\[
%	u(t)=u^{\mathrm{in}}+\int_0^t F(u(s))\,\mathrm ds, \quad t\in[0,T].
%	\]
%	Moreover, $u$ is differentiable on $(0,T)$ and satisfies the differential
%	equation
%	\[
%	\frac{\mathrm du}{\mathrm dt}(t)=F(u(t)), \quad t\in(0,T),
%	\]
%	together with the initial condition $u(0)=u^{\mathrm{in}}$.
%\end{theorem}
\begin{theorem}[\cite{M-N}]\label{TA.1}
	Let $A$ be a nonempty closed convex subset of a Hausdorff locally convex space $X$.
	If $f:A\to A$ is continuous and $f(A)$ is relatively compact in $X$,
	then $f$ has a fixed point in $A$.
\end{theorem}

\vspace{0.2cm}

\begin{theorem}\label{TA.2}
	Let $\mathcal B$ and $\tilde{\mathcal H}$ be the product spaces
	\[
	\mathcal{B}:= \ell_\infty(\mathbb{R}^d)\times \ell_\infty(\mathbb{R}^d),\quad
	\tilde{\mathcal{H}}:= \ell_{\tilde{m}}^2(\mathbb{R}^d)\times \ell_{\tilde{m}}^2(\mathbb{R}^d),
	\]
	where $(\tilde{m}_i)_{i\ge1}$ satisfies \[\sum_{i=1}^\infty \tilde{m}_i<\infty, \tilde{m}_i>0, \ \forall~ i,\]
and $\tilde{\mathcal H}$ is equipped with the standard product Hilbert norm
$\|(x,v)\|_{\tilde{\mathcal H}}=(\|x\|_{\ell^2_{\tilde m}}^2+\|v\|_{\ell^2_{\tilde m}}^2)^{1/2}$.
	Assume that $F:\mathcal{B}\to\mathcal{B}$ satisfies
	\begin{enumerate}
		\item[\textup{(H1)}] (\textup{H\"older continuity on bounded sets}) There exist $\theta\in(0,1]$ such that for every
		$R_0>0$ there is a constant $L_{R_0}>0$ with
		\begin{equation}\label{G-2}
			\|F(\xi)-F(\eta)\|_{\tilde{\mathcal H}}
			\le L_{R_0}\,\|\xi-\eta\|_{\tilde{\mathcal H}}^\theta,
		\end{equation}
		$\forall\,\xi,\eta\in\mathcal B\ \text{with }\|\xi\|_{\mathcal B},\|\eta\|_{\mathcal B}\le R_0.$
		\item[\textup{(H2)}] (\textup{Linear growth in $\mathcal B$}) There exist constants $C_1\ge0$ and $C_2\ge0$ such that
		\begin{equation}\label{G-3}
			\|F(u)\|_{\mathcal B}\le C_1+C_2\|u\|_{\mathcal B},
			\quad \forall\,u\in\mathcal B.
		\end{equation}
	\end{enumerate}
	Then for every $u^{\mathrm{in}}\in\mathcal B$, the Cauchy problem \eqref{G-1} admits a global classical solution $u\in C([0,\infty);\mathcal B)\cap C^1((0,\infty);\tilde{\mathcal H})$.
\end{theorem}
\begin{proof} 
We split the proof in the following six steps. 

\vspace{0.2cm}

\noindent\textbf{Step A (Functional setting and the fixed-point map).}
Fix $T>0$ and $R>0$ to be chosen later. Let $(\tilde m_i)_{i\ge1}\subset(0,\infty)$ satisfy
$\tilde M:=\sum_{i=1}^\infty \tilde m_i<\infty$ and set
\[
\tilde{\mathcal H}_T:=C([0,T];\tilde{\mathcal H}),\quad
\|u\|_{\tilde{\mathcal H}_T}:=\sup_{t\in[0,T]}\|u(t)\|_{\tilde{\mathcal H}}.
\]
Denote $R_0:=\|u^{\mathrm{in}}\|_{\mathcal B}+R$ and define
\[
\mathcal S_{T,R}:=\Big\{u\in\tilde{\mathcal H}_T:u(t)\in\mathcal B,\ \forall~t\in[0,T],\ \sup_{t\in[0,T]}\|u(t)-u^{\mathrm{in}}\|_{\mathcal B}\le R\Big\}.
\]
Define $\mathcal T:\mathcal S_{T,R}\to \tilde{\mathcal H}_T$ by
\[
(\mathcal Tu)(t):=u^{\mathrm{in}}+\int_0^t F(u(s))\,\di s,\quad t\in[0,T].
\]
The integral is taken in $\tilde{\mathcal H}$.
Indeed, (H1) and the uniform $\mathcal B$-bound on
$\mathcal S_{T,R}$ imply that
$F(u(\cdot))\in C([0,T];\tilde{\mathcal H})$.
The continuous coordinate projections therefore give
\[
(\mathcal Tu)_i(t)
=(u^{\mathrm{in}})_i+\int_0^t F_i(u(s))\,\di s.
\]
The $\mathcal B$-estimates below follow by taking
suprema separately over the position and velocity components.
Next we show the closedness and convexity of $\mathcal S_{T,R}$ in $\tilde{\mathcal H}_T$.
Obviously, $\mathcal S_{T,R}$ is nonempty (since $u\equiv u^{\mathrm{in}}\in\mathcal S_{T,R}$) and convex.
To prove that $\mathcal S_{T,R}$ is closed in $\tilde{\mathcal H}_T$, let $u^n\in\mathcal S_{T,R}$ and
$u^n\to u$ in $\tilde{\mathcal H}_T$. Denote $u_i=(x_i,v_i)\in\mathbb R^{2d}$.
Because $\tilde m_i>0$ implies
\[\sup_{t\in[0,T]}\|u_i^n(t)-u_i(t)\|^2\le \tilde m_i^{-1}\|u^n-u\|_{\tilde{\mathcal H}_T}^2.\]
Then for each fixed $i\ge1$,
\[
\lim_{n\to \infty}\sup_{t\in[0,T]}\|u_i^n(t)-u_i(t)\|=0.
\]
Fix $t\in[0,T]$ and write
\[
A_n(t):=\|x^n(t)-x^{\mathrm{in}}\|_\infty,\quad
B_n(t):=\|v^n(t)-v^{\mathrm{in}}\|_\infty.
\]
By the definition of $\mathcal S_{T,R}$, $A_n(t)+B_n(t)\le R$ for all $n$ and $t$. The coordinatewise convergence above and lower semicontinuity of the supremum give
\[
\|x(t)-x^{\mathrm{in}}\|_\infty\le\liminf_{n\to\infty}A_n(t),\quad
\|v(t)-v^{\mathrm{in}}\|_\infty\le\liminf_{n\to\infty}B_n(t).
\]
Consequently, for each $t\in[0,T]$,
\[
\|u(t)-u^{\mathrm{in}}\|_{\mathcal B}\le\liminf_n A_n(t)+\liminf_n B_n(t)
\le\liminf_n(A_n(t)+B_n(t))\le R.
\]
Taking the supremum over $t$ gives $u\in\mathcal S_{T,R}$, and hence $\mathcal S_{T,R}$ is closed in $\tilde{\mathcal H}_T$.

\vspace{0.2cm}

\noindent\textbf{Step B ($\mathcal T$ maps $\mathcal S_{T,R}$ into itself).} If $u\in\mathcal S_{T,R}$, then \[\|u(t)\|_{\mathcal B}\le \|u^{\mathrm{in}}\|_{\mathcal B}+R,\quad\forall\ t\in[0,T].\]
Further, by \eqref{G-3}, we have that for $u\in\mathcal S_{T,R}$, 
\begin{equation*}
	\begin{aligned}
		\|F(u(t))\|_{\mathcal B}&
		%\le \|v(t)\|_\infty+ M2^{\alpha}\|v(t)\|_\infty^{\alpha}\\
		%		&\le \|u(t)\|_{\mathcal{B}}+M2^\alpha\|u(t)\|_{\mathcal{B}}^{\alpha}\\
		%		&\le \|u^{\mathrm{in}}\|_{\mathcal{B}}+R+M2^\alpha\left(\|u^{\mathrm{in}}\|_{\mathcal{B}}+R\right)^{\alpha}\\
		\le C_1+C_2\|u(t)\|_{\mathcal{B}}\\
		&\le C_1+C_2(\|u^{\mathrm{in}}\|_{\mathcal{B}}+R), \quad\forall~ t\in[0,T].
	\end{aligned}
\end{equation*}
Hence,
\begin{equation*}
	\begin{aligned}
		\|(\mathcal Tu)(t)-u^{\mathrm{in}}\|_{\mathcal B}
		&\le \int_0^t \|F(u(s))\|_{\mathcal B}\,\di s\\
		&\le T(C_1+C_2(\|u^{\mathrm{in}}\|_{\mathcal{B}}+R)).
\end{aligned}\end{equation*}
If $C_2=0$, choose any $T>0$ and set $R:=TC_1+1$ to obtain $\mathcal T(\mathcal S_{T,R})\subset\mathcal S_{T,R}$.
Assume $C_2>0$ and choose
\begin{equation}
	T:=\frac{1}{2C_2},\quad R:=\frac{C_1}{C_2}+\|u^{\mathrm{in}}\|_{\mathcal B}+1.
\end{equation}
Then $$T(C_1+C_2(\|u^{\mathrm{in}}\|_{\mathcal{B}}+R))\le R,$$ 
so $\mathcal T(\mathcal S_{T,R})\subset \mathcal S_{T,R}$. 

\vspace{0.2cm}

\noindent\textbf{Step C (Continuity of $\mathcal T$ in $\tilde{\mathcal H}_T$).} We claim that \[\mathcal T \ \text{is continuous on}\ \mathcal S_{T,R}\ 
\text{with respect to} \ \|\cdot\|_{\tilde{\mathcal H}_T}.\]
Let $u^n,u\in\mathcal S_{T,R}$ such that
\[
\|u^n-u\|_{\tilde{\mathcal H}_T}\to0\quad\text{as }n\to\infty.
\]
Since $u^n,u\in\mathcal S_{T,R}$, we have the uniform bound
\[
\sup_{t\in[0,T]}\|u^n(t)\|_{\mathcal B}\le \!\|u^{\mathrm{in}}\|_{\mathcal B}+R,
\
\sup_{t\in[0,T]}\|u(t)\|_{\mathcal B}\le \!\|u^{\mathrm{in}}\|_{\mathcal B}+R.
\]
By \eqref{G-2}, we have
\begin{align*}
	&\sup_{t\in[0,T]}\|F(u^n(t))-F(u(t))\|_{\tilde{\mathcal H}}\le L_{R_0}\|u^n-u\|_{\tilde{\mathcal H}_T}^{\theta}\to0.
\end{align*}
Therefore, as $n\to\infty,$
\begin{align*}
	&\|\mathcal Tu^n-\mathcal Tu\|_{\tilde{\mathcal H}_T}\le \sup_{t\in[0,T]}\int_0^t \|F(u^n(s))-F(u(s))\|_{\tilde{\mathcal H}}\,\di s\\
	&\hspace{1cm}\le T \sup_{s\in[0,T]}\|F(u^n(s))-F(u(s))\|_{\tilde{\mathcal H}}\to0.
\end{align*}
Hence $\mathcal T$ is continuous on $\mathcal S_{T,R}$ in $\tilde{\mathcal H}_T$.

\vspace{0.2cm}

\noindent\textbf{Step D (Relative compactness of $\mathcal T(\mathcal S_{T,R})$) in $\tilde{\mathcal H}_T$.} 
We prove it by the Arzel\`a-Ascoli Theorem. On the one hand, for any $u\in\mathcal S_{T,R}$ and $\forall\ 0\le t_1<t_2\le T$,
\begin{equation*}
	\begin{aligned}
		&\|(\mathcal Tu)(t_2)-(\mathcal Tu)(t_1)\|_{\tilde{\mathcal H}}
		\le \int_{t_1}^{t_2}\|F(u(s))\|_{\tilde{\mathcal H}}\,\di s
		\\
		&\hspace{0.5cm}\le \sqrt{\tilde{M}}\int_{t_1}^{t_2}\|F(u(s))\|_{\mathcal B}\,\di s\\
		&\hspace{0.5cm}\le \sqrt{\tilde{M}}(C_1+C_2(\|u^{\mathrm{in}}\|_{\mathcal{B}}+R))|t_2-t_1|.
	\end{aligned}
\end{equation*}
Hence, the family $\{\mathcal Tu:\ u\in\mathcal S_{T,R}\}$ is equi-continuous in $C([0,T];\tilde{\mathcal H})$.
On the other hand, fix $t\in[0,T]$, then for all $u\in\mathcal S_{T,R}$,
\[
\|(\mathcal Tu)(t)\|_{\mathcal B}\le \|u^{\mathrm{in}}\|_{\mathcal B}+(C_1+C_2(\|u^{\mathrm{in}}\|_{\mathcal{B}}+R))T,
\]
so $\{(\mathcal Tu)(t):u\in\mathcal S_{T,R}\}$ is bounded in $\mathcal B$.
By Lemma~\ref{LC.1} (compact embedding $\mathcal B\hookrightarrow\tilde{\mathcal H}$),
$\{(\mathcal Tu)(t):u\in\mathcal S_{T,R}\}$ is relatively compact in $\tilde{\mathcal H}$. By the Arzel\`a-Ascoli Theorem, $\mathcal T(\mathcal S_{T,R})$ is relatively compact in $\tilde{\mathcal H}_T$.

\vspace{0.2cm}

\noindent\textbf{Step E (Local existence via Schauder-Tychonoff).}
We have shown that $\mathcal S_{T,R}$ is a nonempty closed convex subset of the locally convex space $\tilde{\mathcal H}_T$,
that $\mathcal T:\mathcal S_{T,R}\to\mathcal S_{T,R}$ is continuous (Step~C),
and that $\mathcal T(\mathcal S_{T,R})$ is relatively compact in $\tilde{\mathcal H}_T$ (Step~D).
Hence, by the Schauder-Tychonoff fixed point theorem \ref{TA.1}, $\mathcal T$ admits a fixed point $u\in\mathcal S_{T,R}$:
\[
u(t)=u^{\mathrm{in}}+\int_0^t F(u(s))\,ds,\quad t\in[0,T].
\]
In particular, $u(t)\in\mathcal B$ for all $t\in[0,T]$. For every $i$, the Hilbert-space integral identity gives the coordinatewise identity
$u_i(t_2)-u_i(t_1)=\int_{t_1}^{t_2}F_i(u(s))\,\di s$. Taking suprema in each of the two components yields
\begin{align*}
	\|u(t_2)-u(t_1)\|_{\mathcal B}
	&\le \int_{t_1}^{t_2}\|F(u(s))\|_{\mathcal B}\,\di s
	\le C_{T,R}|t_2-t_1|,
\end{align*}
where $C_{T,R}:=C_1+C_2(\|u^{\mathrm{in}}\|_{\mathcal B}+R)$.
Hence $u\in C([0,T];\mathcal B)$ (in fact, $u$ is Lipschitz in $\mathcal B$).
Moreover, since $F(u(\cdot))\in C([0,T];\tilde{\mathcal H})$, differentiating the above identity yields
$u\in C^1([0,T];\tilde{\mathcal H})$ and $u'(t)=F(u(t))$ in $\tilde{\mathcal H}$ for $t\in(0,T)$.

\vspace{0.2cm}

\noindent\textbf{Step F (Global continuation).}
If $C_2>0$, Step~B provides a local existence time $T_*:=1/(2C_2)$, which is independent of the initial datum. Starting from $u^{\mathrm{in}}$, we obtain a solution on $[0,T_*]$. Taking $u(T_*)\in\mathcal B$ as a new initial datum and repeating the same argument, we extend the solution to $[0,2T_*]$, and so on. Therefore, the solution extends to $[0,nT_*]$ for every $n\in\mathbb N$. If $C_2=0$, then $\|F(u)\|_{\mathcal B}\le C_1$ and Step~B holds for any $T>0$ with $R=TC_1+1$, so the same iteration yields a global solution. The integral identities on the successive intervals give the integral identity on $[0,\infty)$. In particular, the solution is continuous in $\mathcal B$. Since $F(u(\cdot))$ is continuous in $\tilde{\mathcal H}$ on every compact time interval by (H1) and the local $\mathcal B$-bounds, differentiating the integral identity gives $u'(t)=F(u(t))$ in $\tilde{\mathcal H}$ for all $t>0$.
\end{proof}

\vspace{0.3cm}

\section{Global existence of classical solutions to \eqref{A-1}}
For our system \eqref{A-1}, we choose
\begin{align*}
	\begin{aligned}
		F(u)&\!:=\Bigg(v_1,\cdots,v_{i},\cdots,\sum_{j=1}^{\infty}m_{j}H(\|x_{j}\!-x_1\|)\Gamma(v_j-v_1),\\
		&\quad \cdots,\sum_{j=1}^{\infty}m_{j}H(\|x_{j}-x_i\|)\Gamma(v_j-v_i),\cdots\Bigg)
	\end{aligned}
\end{align*}
We define the interaction operator $\mathcal A(x,v):=(\mathcal A_i(x,v))_{i\ge1}$, where
\[
\mathcal A_i(x,v):=\sum_{j=1}^\infty m_j H(\|x_j-x_i\|)\Gamma(v_j-v_i),\quad i\in\mathbb N.
\]
We first present H\"older continuity of $F:\mathcal B\to\tilde{\mathcal H}$ on bounded sets.
\begin{lemma}\label{LB.1}
	Let $\alpha\in(0,1)$ and let $\{m_i\}_{i\ge1}\subset[0,\infty)$ satisfy
	\[
	M:=\sum_{i=1}^\infty m_i\in(0,\infty).
	\]
	Define $\tilde m_i:=m_i+2^{-i}$ and $\tilde M:=\sum_{i=1}^\infty \tilde m_i=M+1$.
	Assume that $H$ and $\Gamma$ satisfy $(\mathcal{A}_1)$ and $(\mathcal{A}_2)$. Then, for every $R_0>0$ there exists a constant $L_{R_0}>0$ such that for all
	$u=(x,v),\,\bar u=(\bar x,\bar v)\in\mathcal B$ with
	$\|u\|_{\mathcal B},\|\bar u\|_{\mathcal B}\le R_0$,
	\begin{equation*}
		\|F(u)-F(\bar u)\|_{\tilde{\mathcal H}}
		\le L_{R_0}\,\|u-\bar u\|_{\tilde{\mathcal H}}^{\alpha}.
	\end{equation*}
	In particular, on bounded subsets of $\mathcal B$, the map $F$ is H\"older continuous with exponent $\alpha$ with respect to the norm of $\tilde{\mathcal H}$.
\end{lemma}
\begin{proof} Let $u=(x,v)$ and $\bar u=(\bar x,\bar v)$ satisfy $\|u\|_{\mathcal B},\|\bar u\|_{\mathcal B}\le R_0$.
Denote $\Delta x:=x-\bar x$, $\Delta v:=v-\bar v$, and $\Delta u:=u-\bar u$.
Then
\[
\|\Delta x\|_\infty\le 2R_0,\ \|\Delta v\|_\infty\le 2R_0,
\
\|\Delta u\|_{\tilde{\mathcal H}}\le 2\sqrt{\tilde M}\,R_0=:S_0.
\]
Recall that $F(u)=(v,\mathcal A(x,v))$. Hence
\begin{equation}\label{H-2}
	\begin{aligned}
		\|F(u)-F(\bar u)\|_{\tilde{\mathcal H}}
		&\le \|v-\bar v\|_{\ell_{\tilde m}^2}+\|\mathcal A(x,v)-\mathcal A(\bar x,\bar v)\|_{\ell_{\tilde m}^2}=: I_1+I_2.\end{aligned}
\end{equation}

\noindent$\bullet$~\textbf{Step A. Estimate of $I_1$.}
Since $\alpha\in(0,1)$ and $\|\Delta u\|_{\tilde{\mathcal H}}\le S_0$, we have
\begin{equation}\label{H-3}
	I_1=\|\Delta v\|_{\ell_{\tilde m}^2}\le \|\Delta u\|_{\tilde{\mathcal H}}
	\le S_0^{1-\alpha}\|\Delta u\|_{\tilde{\mathcal H}}^\alpha.
\end{equation}

\noindent$\bullet$~\textbf{Step B. Estimate of $I_2$.}
Let $\|H\|_{\mathrm{Lip}}$ be the Lipschitz constant of $H$.
Recall $\Gamma:\mathbb R^d\to\mathbb R^d$ is the componentwise sign-power map with exponent $\alpha\in(0,1)$.
Then there exists a constant $C_\Gamma:=2^{1-\alpha}d^{\frac{1-\alpha}{2}}$ such that $ \forall~ a,b\in\mathbb R^d$,
\begin{equation}\label{H-4}
	\|\Gamma(a)-\Gamma(b)\|\le C_\Gamma \|a-b\|^\alpha,\
	\|\Gamma(a)\|\le C_\Gamma \|a\|^\alpha.
\end{equation}
Since $\|v\|_\infty,\|\bar v\|_\infty\le R_0$, for all $i,j$,
\begin{equation}\label{H-5}
	\|\Gamma(v_j-v_i)\|\le C_\Gamma (2R_0)^\alpha,
	\
	\|\Gamma(\bar v_j-\bar v_i)\|\le C_\Gamma (2R_0)^\alpha.
\end{equation}
For each $i\ge1$, write
\begin{align*}
	\mathcal A_i(x,v)-\mathcal A_i(\bar x,\bar v)
	&=\sum_{j\ge1} m_j\Big(H_{ij}\Gamma_{ij}-\bar H_{ij}\bar\Gamma_{ij}\Big)\\
	&=:\sum_{j\ge1} m_jB_{ij}+\sum_{j\ge1} m_jC_{ij}=:B_i+C_i,
\end{align*}
where $H_{ij}:=H(\|x_j-x_i\|)$, $\bar H_{ij}:=H(\|\bar x_j-\bar x_i\|)$,
$\Gamma_{ij}:=\Gamma(v_j-v_i)$ and $\bar\Gamma_{ij}:=\Gamma(\bar v_j-\bar v_i)$, and
\[
B_{ij}:=(H_{ij}-\bar H_{ij})\Gamma_{ij},\quad
C_{ij}:=\bar H_{ij}(\Gamma_{ij}-\bar\Gamma_{ij}).
\]
\vspace{0.1cm}

\noindent$\bullet$~\textbf{Step B-1. Bound for $B_i$.}
By the Lipschitz continuity of $H$, the reverse triangle inequality
$\big|\|x_j-x_i\|-\|\bar x_j-\bar x_i\|\big|\le\|\Delta x_j\|+\|\Delta x_i\|$, and \eqref{H-5}, we obtain
\[
\|B_{ij}\|
\le \|H\|_{\mathrm{Lip}}(\|\Delta x_j\|+\|\Delta x_i\|)\, C_\Gamma(2R_0)^\alpha.
\]
Hence,
\begin{align*}
	\|B_i\|
	&\le C_x\sum_{j\ge1}m_j(\|\Delta x_j\|+\|\Delta x_i\|)= C_x\Big(\sum_{j\ge1}m_j\|\Delta x_j\|+M\|\Delta x_i\|\Big),
\end{align*}
where $C_x:=\|H\|_{\mathrm{Lip}}\,C_\Gamma(2R_0)^\alpha$.
Since $m_j\le \tilde m_j$, by Cauchy-Schwarz's inequality,
\begin{align*}
	\sum_{j\ge1}m_j\|\Delta x_j\|&\le
	\Big(\sum_{j\ge1}m_j\Big)^{1/2}\Big(\sum_{j\ge1}m_j\|\Delta x_j\|^2\Big)^{1/2}\le \sqrt{M}\,\|\Delta x\|_{\ell^2_{\tilde m}}.
\end{align*}
Thus
\begin{equation}\label{H-6}
	\|B_i\|\le C_x\Big(\sqrt{M}\,\|\Delta x\|_{\ell^2_{\tilde m}}+M\|\Delta x_i\|\Big).
\end{equation}

\noindent$\bullet$~\textbf{Step B-2. Bound for $C_i$.}
Since $0\le \bar H_{ij}\le \|H\|_\infty$, by \eqref{H-4},
\[
\|C_{ij}\|
\le  \|H\|_\infty C_\Gamma\,\|\Delta v_j-\Delta v_i\|^\alpha.
\]
Using $(a+b)^\alpha\le a^\alpha+b^\alpha$ for $a,b\ge0$ and $\alpha\in(0,1)$, we get
\[
\|\Delta v_j-\Delta v_i\|^\alpha\le \|\Delta v_j\|^\alpha+\|\Delta v_i\|^\alpha,
\]
hence
\begin{equation}\label{H-7}
	\begin{aligned}
		\|C_i\|
		&\le  C_v\sum_{j\ge1}m_j(\|\Delta v_j\|^\alpha+\|\Delta v_i\|^\alpha)\le  C_v\Big(M^{1-\alpha/2}\|\Delta v\|_{\ell^2_{\tilde m}}^\alpha+M\|\Delta v_i\|^\alpha\Big),
	\end{aligned}
\end{equation}
where $C_v:=\|H\|_\infty C_\Gamma$. Here we use the fact from the Lyapunov's inequality on finite measure space $(\mathbb N,m)$ that
\begin{align*}
	\sum_{j\ge1}m_j\|\Delta v_j\|^\alpha
	&\le M^{1-\alpha/2}\Big(\sum_{j\ge1}m_j\|\Delta v_j\|^2\Big)^{\alpha/2}\le M^{1-\alpha/2}\|\Delta v\|_{\ell^2_{\tilde m}}^\alpha.
\end{align*}

\vspace{0.1cm}

\noindent$\bullet$~\textbf{Step B-3. Estimate of $I_2$.}
By Minkowski's inequality in $\ell^2_{\tilde m}$, we have
\begin{equation}\label{H-8}
	I_2=\|\mathcal A(x,v)-\mathcal A(\bar x,\bar v)\|_{\ell^2_{\tilde m}}
	\le \|B\|_{\ell^2_{\tilde m}}+\|C\|_{\ell^2_{\tilde m}},
\end{equation}
where $B=(B_i)_{i\ge1}$ and $C=(C_i)_{i\ge1}$.

\vspace{0.2cm}

\noindent$\diamond$ (Bound for $\|B\|_{\ell^2_{\tilde m}}$).
From \eqref{H-6} and $(a+b)^2\le 2(a^2+b^2)$,
\begin{align*}
	\|B\|_{\ell^2_{\tilde m}}
	&=\left(\sum_{i\ge1}\tilde m_i\|B_i\|^2\right)^{\frac{1}{2}}\le \sqrt{2}C_x\left(\sum_{i\ge1}\tilde m_i\Big(M\|\Delta x\|_{\ell^2_{\tilde m}}^2+M^2\|\Delta x_i\|^2\Big)\right)^{\frac{1}{2}}\\
    &\hspace{0.4cm}=\sqrt{2}\,C_x\sqrt{M\tilde M+M^2}\,\|\Delta x\|_{\ell^2_{\tilde m}}.
\end{align*}
\noindent$\diamond$ (Bound for $\|C\|_{\ell^2_{\tilde m}}$).
From \eqref{H-7},
\begin{align*}
	\|C\|_{\ell^2_{\tilde m}}^2
	&\le 2C_v^2\sum_{i\ge1}\tilde m_i\Big(M^{2-\alpha}\|\Delta v\|_{\ell^2_{\tilde m}}^{2\alpha}+M^2\|\Delta v_i\|^{2\alpha}\Big)\\
	&=2C_v^2\Big(M^{2-\alpha}\tilde M\|\Delta v\|_{\ell^2_{\tilde m}}^{2\alpha}
	+M^2\sum_{i\ge1}\tilde m_i\|\Delta v_i\|^{2\alpha}\Big).
\end{align*}
Since $2\alpha<2$ and $\tilde M<\infty$, Lyapunov's inequality on $(\mathbb N,\tilde m)$ gives
\[
\sum_{i\ge1}\tilde m_i\|\Delta v_i\|^{2\alpha}
\le\tilde M^{1-\alpha}\|\Delta v\|_{\ell^2_{\tilde m}}^{2\alpha}.
\]
Therefore,
\begin{equation*}
	\|C\|_{\ell^2_{\tilde m}}
	\le \sqrt{2}\,C_v\sqrt{M^{2-\alpha}\tilde M+M^2\tilde M^{1-\alpha}}\ \|\Delta v\|_{\ell^2_{\tilde m}}^{\alpha}.
\end{equation*}
We go back to \eqref{H-8} and obtain
\begin{equation}\label{H-9}
	I_2\le K_1\|\Delta x\|_{\ell^2_{\tilde m}}+K_2\|\Delta v\|_{\ell^2_{\tilde m}}^\alpha,
\end{equation}
with
\[
K_1:=\sqrt{2}\,C_x\sqrt{M\tilde M+M^2},\quad K_2:=\sqrt{2}\,C_v\sqrt{M^{2-\alpha}\tilde M+M^2\tilde M^{1-\alpha}}.
\]
Finally, since $\|\Delta x\|_{\ell^2_{\tilde m}}\le \|\Delta u\|_{\tilde{\mathcal H}}$ and $\|\Delta u\|_{\tilde{\mathcal H}}\le S_0$,
\[
\|\Delta x\|_{\ell^2_{\tilde m}}
\le S_0^{1-\alpha}\|\Delta u\|_{\tilde{\mathcal H}}^\alpha,
\quad
\|\Delta v\|_{\ell^2_{\tilde m}}^\alpha\le \|\Delta u\|_{\tilde{\mathcal H}}^\alpha.
\]
Thus, from \eqref{H-9},
\begin{equation}\label{H-10}
	I_2\le \big(K_1S_0^{1-\alpha}+K_2\big)\|\Delta u\|_{\tilde{\mathcal H}}^\alpha.
\end{equation}

\vspace{0.1cm}

\noindent$\bullet$~\textbf{Step C. Conclusion.}
Combining \eqref{H-2}, \eqref{H-3} and \eqref{H-10}, we obtain
\begin{align*}
	\|F(u)-F(\bar u)\|_{\tilde{\mathcal H}}
	&\le \Big(S_0^{1-\alpha}+K_1S_0^{1-\alpha}+K_2\Big)\|\Delta u\|_{\tilde{\mathcal H}}^\alpha=:L_{R_0}\|\Delta u\|_{\tilde{\mathcal H}}^\alpha,
\end{align*}
where $L_{R_0}>0$ depends only on $R_0$ and the parameters
$M,\tilde M,d,\|H\|_{\mathrm{Lip}},\|H\|_\infty,\alpha$.
\end{proof}

\vspace{0.2cm}

\begin{lemma}\label{LB.2}
	Assume that $H$ and $\Gamma$ satisfy $(\mathcal{A}_1)$ and $(\mathcal{A}_2)$. Then the map $F:\mathcal B\to\mathcal B$ maps bounded sets into bounded sets, in particular, 
	\begin{equation*}
		\|F(u)\|_{\mathcal{B}}\le \|v\|_\infty+ \|H\|_\infty Md^{\frac{1-\alpha}{2}}2^{\alpha}\|v\|_\infty^{\alpha},\quad \forall~ u=(x,v)\in \mathcal{B}.
	\end{equation*}
\end{lemma}
\begin{proof} We prove the stated growth bound. Take any $u=(x,v)\in\mathcal B$.
Then, for the first part of $F$ which is $F_1(u)=v$, we have
\[
\|F_1(u)\|_\infty=\|v\|_\infty.
\]
Next, we consider the second component $F_2(u)$ which is a nonlinear interaction term:
\[
F_{2}(u)=\left(\sum_{j=1}^{\infty}m_{j}H(\|x_{j}-x_i\|)\Gamma(v_j-v_i)\right)_{i\in\mathbb{N}}.
\]
Note for all $i,j\in\mathbb N$,
\[
\|v_j-v_i\|\le \|v_j\|+\|v_i\|\le2\|v\|_\infty.
\]
Hence, since $0\le H\le
\|H\|_\infty$ and
$\Gamma$ is defined as \eqref{A-2}, it follows that for each $i\in\mathbb N$,
\begin{align*}
	\|(F_2(u))_i\|
	&\le \sum_{j=1}^\infty m_j\,|H(\|x_j-x_i\|)|\,\|\Gamma(v_j-v_i)\| \\
	&\le \|H\|_\infty\sum_{j=1}^\infty m_j\left(\sum_{k=1}^{d}|v_j^{(k)}-v_i^{(k)}|^{2\alpha}\right)^{\frac{1}{2}}\\
	&\le\|H\|_\infty\sum_{j=1}^\infty m_jd^{\frac{1-\alpha}{2}}\|v_j-v_i\|^{\alpha}\\
	&\le \|H\|_\infty Md^{\frac{1-\alpha}{2}}2^{\alpha}\|v\|_\infty^{\alpha}.
\end{align*}
We take the supremum over $i$ to get 
\[
\|F_2(u)\|_\infty\le \|H\|_\infty Md^{\frac{1-\alpha}{2}}2^{\alpha}\|v\|_\infty^{\alpha}.
\]
Therefore, we have
\[
\|F(u)\|_{\mathcal B}
\le \|v\|_\infty+ \|H\|_\infty Md^{\frac{1-\alpha}{2}}2^{\alpha}\|v\|_\infty^{\alpha}.
\]
\end{proof}

\vspace{0.2cm}

\begin{corollary} \label{global-existence-model}
	Let $\alpha\in(0,1)$ and let $\{m_i\}_{i\ge1}\subset[0,\infty)$ satisfy
	\[
	M:=\sum_{i=1}^\infty m_i\in(0,\infty).
	\]
	Assume that $H:[0,\infty)\to[0,\infty)$ and $\Gamma$ satisfy $(\mathcal A_1)$ and $(\mathcal A_2)$.
	Then, for any initial datum $u^{\mathrm{in}}=(x^{\rm in},v^{\rm in})\in\mathcal B$, the Cauchy problem
	\[
	\dot u(t)=F(u(t)),\quad u(0)=u^{\mathrm{in}},
	\]
admits a global classical solution
\[
u\in C([0,\infty);\mathcal B)\cap C^1((0,\infty);\tilde{\mathcal H}).
\]
\end{corollary}
\begin{proof}
	Set $\tilde m_i:=m_i+2^{-i}$, $\tilde M:=\sum_{i=1}^\infty\tilde m_i=M+1$, and
	\[
	\tilde{\mathcal H}:=\ell_{\tilde m}^2(\mathbb R^d)\times\ell_{\tilde m}^2(\mathbb R^d).
	\]
	We verify the assumptions (H1) and (H2) in Theorem~\ref{TA.2}. By Lemma~\ref{LB.1}, for every $R_0>0$, there exists a constant $L_{R_0}>0$ such that
	\[
	\|F(u)-F(\bar u)\|_{\tilde{\mathcal H}}\le L_{R_0}\|u-\bar u\|_{\tilde{\mathcal H}}^\alpha,
	\]
	for all $u,\bar u\in\mathcal B$ with $\|u\|_{\mathcal B},\|\bar u\|_{\mathcal B}\le R_0$. Thus, (H1) holds with $\theta=\alpha$.
	On the other hand, by Lemma~\ref{LB.2} and the inequality $y^\alpha\le 1+y$ for $y\ge0$, we have
	\[
	\|F(u)\|_{\mathcal B}\le \|H\|_\infty Md^{\frac{1-\alpha}{2}}2^\alpha+(1+\|H\|_\infty Md^{\frac{1-\alpha}{2}}2^\alpha)\|v\|_\infty.
	\]
	Since $\|v\|_\infty\le\|u\|_{\mathcal B}$, (H2) holds with
	\[
	C_1=\|H\|_\infty Md^{\frac{1-\alpha}{2}}2^\alpha,\quad C_2=1+C_1.
	\]
	Therefore, we apply Theorem~\ref{TA.2} to obtain a global classical solution
	\[
	u=(x,v)\in C([0,\infty);\mathcal B)\cap C^1((0,\infty);\tilde{\mathcal H})
	\]
	satisfying $\dot u(t)=F(u(t))$ for all $t>0$ and $u(0)=u^{\rm in}$. Since $\tilde m_i>0$ for every $i\in\mathbb N$, the coordinate projections are continuous in $\tilde{\mathcal H}$. Hence, $(x_i,v_i)$ satisfies \eqref{A-1} for every $i\in\mathbb N$.
\end{proof}

\vspace{0.3cm}

\section{Compact embedding}  
\begin{lemma}[Compact embedding]\label{LC.1}
	Let $\{m_i\}_{i\ge1}\subset [0,\infty)$ satisfy
	\[
	M:=\sum_{i=1}^\infty m_i<\infty.
	\]
	Assume in addition that $m_i>0$ for every $i$. Then the inclusion into the weighted space is compact.
	Consequently, the embedding
	\[
	\mathcal B=\ell_\infty(\mathbb R^d)\times \ell_\infty(\mathbb R^d)
	\ \hookrightarrow\
	\mathcal H=\ell_m^2(\mathbb R^d)\times \ell_m^2(\mathbb R^d)
	\]
	is compact.
\end{lemma}
\begin{proof}
We first prove compactness of the embedding $\ell_\infty(\mathbb R^d)\hookrightarrow \ell_m^2(\mathbb R^d)$.
Let $\{v^n\}_{n\in\mathbb N}\subset \ell_\infty(\mathbb R^d)$ be bounded, i.e., there exists $R>0$ such that
\[
\sup_{n\in\mathbb N}\|v^n\|_\infty \le R.
\]
If $R=0$, the conclusion is immediate; hence assume $R>0$.
Write $v^n=(v_i^n)_{i\ge1}$. For each fixed $i\ge1$, the sequence $\{v_i^n\}_{n}$ is bounded in $\mathbb R^d$,
hence admits a convergent subsequence. By a diagonal argument, there exist a subsequence (still denoted by $\{v^n\}$)
and $v=(v_i)_{i\ge1}\in \ell_\infty(\mathbb R^d)$ such that
\[
v_i^n \to v_i\quad \text{as }n\to\infty,\quad \forall~ i\ge1.
\]
We claim that $v^n\to v$ strongly in $\ell_m^2(\mathbb R^d)$.
Fix $\varepsilon>0$. Choose $N\in\mathbb N$ such that
\[
\sum_{i>N} m_i < \frac{\varepsilon^2}{16R^2}.
\]
Then, we get that for all $n$,
\[
\sum_{i>N} m_i\|v_i^n-v_i\|^2
\le \sum_{i>N} m_i (2R)^2
=4R^2\sum_{i>N}m_i
<\frac{\varepsilon^2}{4}.
\]
On the other hand, since $v_i^n\to v_i$ for each $1\le i\le N$ and the sum is finite, we have
\[
\lim_{n\to \infty}\sum_{i=1}^N m_i\|v_i^n-v_i\|^2 =0.\]
Hence there exists $n_0$ such that for all $n\ge n_0$,
\[
\sum_{i=1}^N m_i\|v_i^n-v_i\|^2 < \frac{\varepsilon^2}{4}.
\]
Combining the last two estimates yields, for all $n\ge n_0$,
\[
\|v^n-v\|_{\ell_m^2}^2
=\sum_{i=1}^\infty m_i\|v_i^n-v_i\|^2
<\frac{\varepsilon^2}{4}+\frac{\varepsilon^2}{4}
=\frac{\varepsilon^2}{2},
\]
which implies $\|v^n-v\|_{\ell_m^2}<\varepsilon$ for $n$ large enough. This proves that every bounded sequence in
$\ell_\infty(\mathbb R^d)$ has a strongly convergent subsequence in $\ell_m^2(\mathbb R^d)$, i.e., the embedding is compact. Finally, for the product spaces $\mathcal B$ and $\mathcal H$, let $u^n=(x^n,v^n)$ be bounded in $\mathcal B$.
Applying the above compactness to $\{x^n\}$ and $\{v^n\}$ separately and taking a common subsequence, we obtain
$x^n\to x$ and $v^n\to v$ strongly in $\ell_m^2(\mathbb R^d)$. Hence $u^n\to (x,v)$ strongly in $\mathcal H$.
Therefore, $\mathcal B\hookrightarrow \mathcal H$ is compact.
\end{proof}

\end{document}